%% file: main.tex
\documentclass[11pt,reqno]{amsart}
\usepackage{amsmath,amssymb,amsthm}
\usepackage{amsfonts}
\usepackage{mathrsfs}
\usepackage{color}  
\usepackage[all,cmtip]{xy}
\usepackage{tikz}
\usetikzlibrary{arrows,calc, intersections, decorations.pathmorphing,backgrounds,positioning,fit, through}
\colorlet{LightBlue}{blue!50}

\usepackage[margin=1.3in]{geometry}
\usepackage{hyperref}   

\newtheorem{theorem}{Theorem}[section]
\newtheorem{proposition}[theorem]{Proposition}
\newtheorem{corollary}[theorem]{Corollary}
\newtheorem{lemma}[theorem]{Lemma}

\theoremstyle{definition}
\newtheorem{definition}[theorem]{Definition}
\newtheorem{example}[theorem]{Example}

\theoremstyle{remark}
\newtheorem{remark}[theorem]{Remark}

\numberwithin{equation}{section}

\newcommand{\thmref}[1]{Theorem~\ref{#1}}

\newcommand{\lemref}[1]{Lemma~\ref{#1}}
\newcommand{\corref}[1]{Corollary~\ref{#1}}

\newcommand{\exref}[1]{Example~\ref{#1}}

\newcommand{\calH}{{\mathcal H}}

\newcommand{\calL}{{\mathcal L}}

\newcommand{\calP}{{\mathcal P}}

\newcommand{\calS}{{\mathcal S}}

\newcommand{\calX}{{\mathcal X}}

\newcommand{\CH}{\mathcal{CH}}

\newcommand{\F}{{\mathbb F}}

\newcommand{\R}{{\mathbb R}}

\newcommand{\from}{\colon \thinspace}

\newcommand{\param}{{\mathchoice{\mkern1mu\mbox{\raise2.2pt\hbox{$
\centerdot$}}
\mkern1mu}{\mkern1mu\mbox{\raise2.2pt\hbox{$\centerdot$}}\mkern1mu}{
\mkern1.5mu\centerdot\mkern1.5mu}{\mkern1.5mu\centerdot\mkern1.5mu}}}

\DeclareMathOperator{\Aut}{Aut}
\DeclareMathOperator{\Inn}{Inn}

\DeclareMathOperator{\MCG}{Map}
\DeclareMathOperator{\Homeo}{Homeo}

\DeclareMathOperator{\out}{Out}
\DeclareMathOperator{\Core}{Core}
\DeclareMathOperator{\Shadow}{Shadow}

\newcommand{\bM}{\underline{M}}
\newcommand{\bS}{\underline{S}}
\newcommand{\bX}{\underline{X}}
\newcommand{\balpha}{\underline{\alpha}}
\newcommand{\bbeta}{\underline{\beta}}

\newcommand{\bSigma}{\underline{\Sigma}}

\newcommand{\cv}{\calX_{\rm fill}}
\newcommand{\sphere}{\calS_{\rm fill}}

\newcommand{\Out}{\out(\F_n)}  

\begin{document}

\title{Uniform fellow traveling between surgery paths in the sphere graph}

  \author[M.~Clay]{Matt Clay}
  \address{Dept. of Mathematical Sciences\\ 
  University of Arkansas\\
  Fayetteville, Arkansas, USA 72701}
  \email{\tt mattclay@uark.edu}
  \thanks{The first author was partially funded by the Simons Foundation.}

  \author[Y.~Qing]{Yulan Qing}
  \address{Dept. of Mathematics\\ 
  University of Toronto\\
  Toronto, Ontario, Canada M5S 2E4}
  \email{\tt yulan.qing@utoronto.ca}
  
  \author[K.~Rafi]{Kasra Rafi}
  \address{Dept. of Mathematics\\ 
  University of Toronto\\
  Toronto, Ontario, Canada M5S 2E4}
  \email{\tt rafi@math.toronto.edu}
  \thanks{The second author was partially funded by NSERC Discovery grant 
  RGPIN 435885.}  
\date{\today}

\begin{abstract}
We show that the Hausdorff distance between any forward and any backward surgery paths
in the sphere graph is at most 2.  From this it follows that the Hausdorff distance between 
any two surgery paths with the same initial sphere system and same target sphere system 
is at most 4.  Our proof relies on understanding how surgeries affect the Guirardel core
associated to sphere systems. We show that applying a surgery is equivalent to 
performing a Rips move on the Guirardel core. 
\end{abstract}


\maketitle

\input{intro.tex}
\input{prelim.tex}

\input{sphere.tex}

\input{surgery.tex}

\input{proof.tex}

\bibliography{bib}
\bibliographystyle{alpha}

\end{document}

%% file: intro.tex

\section{Introduction}\label{sec:intro}

In this paper, we study the surgery paths in the sphere graph.  Let $\bM$ be the connected sum of $n$ copies of $S^1 \times S^2$ (we reserve the notation $M$ for the universal cover of $\bM$ which is used more frequently in the body of the paper).  The vertices of the sphere graph are essential sphere systems in $\bM$ and edges encode containment (see Section~\ref{sec:sphere-splittings} for precise definitions). We denote the sphere graph by $\calS$ and the associated metric with $d_\calS$. It is known that the sphere graph $(\calS, d_\calS)$ is hyperbolic in the sense of Gromov \cite{mosher:HS, un:HH}.  The relationship between the optimal hyperbolicity constant and the rank of the fundamental group of $\bM$ (which is isomorphic to $\F_{n}$, the free group of rank $n$) is unknown.

Given a pair of (filling) sphere systems $\bS$ and $\bSigma$, there is a natural family of paths, called surgery paths, connecting them.  They are obtained by replacing larger and larger portions of spheres in $\bS$ with pieces of spheres in $\bSigma$. This process is not unique. Also, families of paths that start from $\bS$ with target $\bSigma$ are different from those starting from $\bSigma$ with target $\bS$.  It follows from \cite{un:HH} that surgery paths are quasi-geodesics.  Together with the hyperbolicity of the sphere graph, this implies that different surgery paths starting with $\bS$ and with target $\bSigma$ have bounded Hausdorff distance in the sphere graph.  The bound depends on the optimal hyperbolicity constant, which as stated above, does not have a good qualitative estimate.   

However, in this paper we show that, in any rank, any two surgery paths are within Hausdorff 
distance at most $4$ of each other.  This follows by comparing a surgery path that starts from $\bS$ with target $\bSigma$ to a surgery path starting from $\bSigma$ with target $\bS$. 

\begin{theorem} \label{Thm:Main}
Let $\bS$ and $\bSigma$ be two filling sphere systems and let 
\begin{equation*}
\bS= \bS_1, \bS_2, \ldots, \bS_m,  
\qquad  d_{\calS}(\bS_{m},\bSigma) \leq 1
\end{equation*}
be a surgery sequence starting from $\bS$ towards $\bSigma$ and 
\begin{equation*}
\bSigma= \bSigma_1, \bSigma_2, \ldots, \bSigma_{\mu},
\qquad d_{\calS}(\bSigma_{\mu},\bS) \leq 1
\end{equation*}
be a surgery sequence in the opposite direction. Then, for every $\bS_i$ there is a 
$\bSigma_j$ so that $d_\calS(\bS_i, \bSigma_j) \leq 2$.  
\end{theorem}

Using this, we get the bound of 4 between paths with the same initial sphere system and same target sphere system.

\begin{theorem}\label{Thm:Main 2}
Let $\bS$ and $\bSigma$ be two filling sphere systems and let 
\begin{align*}
\bS &= \bS_1, \bS_2, \ldots, \bS_m ,
\qquad  d_{\calS}(\bS_{m},\bSigma) \leq 1 \\
\intertext{and}
\bS &= \bS'_1, \bS'_2, \ldots, \bS'_n,
\qquad d_{\calS}(\bS'_{n},\bSigma) \leq 1
\end{align*}
be two surgery sequences starting from $\bS$ towards $\bSigma$.  Then, for every $\bS_i$ there is a $\bS'_j$ so that $d_\calS(\bS_i, \bS'_j) \leq 4$.  
\end{theorem}

\begin{proof}
Fix two filling sphere systems $\bS$ and $\bSigma$ and surgery paths as in the statement of the theorem.  Let
\begin{equation*}
\bSigma= \bSigma_1, \bSigma_2, \ldots, \bSigma_{\mu}
\qquad d_{\calS}(\bSigma_{\mu},\bS) \leq 1
\end{equation*}
be a surgery sequence starting at $\bSigma$ towards $\bS$.  Given $\bS_{i}$, by Theorem~\ref{Thm:Main}, there is a $\bSigma_{k}$ such that $d_{\calS}(\bS_{i},\bSigma_{k}) \leq 2$.  Applying Theorem~\ref{Thm:Main} again, there is a $\bS'_{j}$ such that $d_{\calS}(\bSigma_{k},\bS_{j}) \leq 2$.  Thus $d_{\calS}(\bS_{i},\bS'_{j}) \leq 4$ as desired.
\end{proof}

The sphere graph is a direct analogue of the graph of arcs on a surface with boundary.  In fact, there is an embedding of the arc graph into the sphere graph. The arc graph is known to be uniformly hyperbolic \cite{aougab:UH, bowditch:UH, rafi:US, piotr:UH, sisto:AH}.  
Since the solution of many algorithmic problems for mapping class groups or
hyperbolic $3$--manifold that fibers over a circle rely on the action of mapping class 
group on various curve and arc complexes, the uniform hyperbolicity clarifies
which constant depend on the genus and which ones are genus independent. 
The uniform hyperbolicity of the sphere graph (or one of the other combinatorial complexes associated to $\Out$) is a central open question in the study of the group $\Out$, the group of outer automorphisms of the free group.  Note that \thmref{Thm:Main 2} is not sufficient to prove that the sphere graph is uniformly hyperbolic. 

\subsection*{Summary of other results}
The Guirardel core \cite{ar:Guirardel05} is a square complex associated to two trees 
equipped with isometric actions by a group, in our case $\F_{n}$.  This is an analogue of a 
quadratic differential in the surface case; 
the area of the core is the intersection number between the two associated sphere systems.  
Following \cite{ar:BBC10}, in Section~\ref{sec:core}, we describe how to compute the 
core using the change of marking map between the two trees. \lemref{lem:hull} gives 
a simple condition when a product of two edges is in the core, which will be used in 
future work to study the core. Also, in Section~\ref{sec:sphere-core} we define the
core for two sphere systems, $\Core(\bS, \bSigma)$ directly using the intersection 
pattern of the spheres and show this this object is isomorphic to the Guirardel core for 
the associated tree (\thmref{th:core to core}). Much of what is contained in these 
two sections is known to the experts, however, we include a self-contained exposition 
of the material since it is not written in an easily accessible way in the literature. 

Applying a surgery to a sphere system amounts to applying a splitting
move to the dual tree (see \exref{ex:rips}), however, not all splittings towards a given tree
come from surgeries. In general, applying a splitting move 
could change the associated core in unpredictable ways potentially increasing 
the volume of the core. We will show that (\thmref{thm:rips=surgery}) applying a surgery 
is equivalent to performing a Rips move on the Guirardel core. 
That is, there is a subset of all splitting paths between two trees that is natural from the 
point of view of the Guirardel core and it matches exactly with the set of splitting sequences 
that are associated to surgery paths. 

\subsection*{Outline of the proof} 
Our proof of Theorem~\ref{Thm:Main} analyses Guirardel core $\Core(\bS_i, \bSigma_j)$.
Generally, this does not have to be related to $\Core(\bS, \bSigma)$.
However, we show that, for small values of $i$ and $j$, the spheres 
$\bS_i$ and $\bSigma_j$ are still in normal form and $\Core(\bS_i, \bSigma_j)$ can 
be obtained from $\Core(\bS, \bSigma)$ via a sequence of vertical and horizontal Rips 
moves (Proposition~\ref{Prop:Isomorphism}).  For every $i$, there is the smallest $j$
where this breaks down which is exactly the moment the surgery path from $\bSigma$
to $\bS$ passes near $\bS_i$. The proof of Theorem~\ref{Thm:Main} is completed in 
Section~\ref{Sec:Proof}: 
for every $\bS_i$, apply enough surgery on $\bSigma$ until $\Core(\bS_i, \bSigma_j)$ has 
a free edge, which implies $d_\calS(\bS_i, \bSigma_j)\leq 2$.

\subsection*{Acknowledgements}
We would like to thank Mladen Bestvina for helpful conversations, particularly, 
for giving us the idea to examine how the core changes along splitting sequence. 

%% file: prelim.tex

\section{Sphere systems and free splittings}\label{sec:sphere-splittings} 

Let $\bM$ be the connected sum of $n$ copies of $S^1 \times S^2$ and fix an identification 
of $\pi_1(\bM)$ with $\F_n$.  There is a well-known correspondence between spheres 
in $\bM$ and graph of group decompositions of $\F_{n}$ with trivial edge groups.  
We explain this correspondence now. 

\begin{definition}\label{def:sphere}
A \emph{sphere system} $\bS \subset \bM$ is a finite union of disjoint essential (does not bound a 3--ball) embedded 2--spheres in $\bM$.  We specifically allow for the possibility that a sphere system contains parallel, i.e., isotopic, spheres.  A sphere system is \emph{filling} if each of the complementary regions $\bM - \bS$ are simply-connected.    

We define a preorder on the set of sphere systems by $\bS \preceq \bSigma$ if every sphere in $\bS$ is isotopic to a sphere in $\bSigma$.  This induces an equivalence relation: $\bS \sim \bSigma$ if $\bS \preceq \bSigma$ and $\bSigma \preceq \bS$.  The set of equivalence classes of sphere systems in $\bM$ is denoted by $\calS$; the subset of equivalence classes of filling sphere systems is denoted by $\sphere$.  When there can be no confusion, we denote the equivalence class of $\bS$ again by $\bS$.    

The preorder induces a partial order on $\calS$ that we continue to denote by $\preceq$.  The \emph{sphere graph} is the simplicial graph with vertex set $\calS$ and edges corresponding to
domination $\bS \preceq \bSigma$.  For $\bS, \bSigma \in \calS$, we denote by $d_{\calS}(\bS,\bSigma)$ the distance between $\bS$ and $\bSigma$ in the sphere graph.  This the fewest number of edges in an edge path between the two vertices.
\end{definition}    

We denote the universal cover of $\bM$ by $M$ and the lift of the sphere system $\bS$ to $M$ by $S$; we will refer to $S$ as \emph{a sphere system in $M$}.  To simplify notation, we use $\calS$ and respectively $\sphere$ to denote (equivalence classes of) sphere systems, respectively filling sphere systems, in $M$.  Let $\displaystyle \MCG(\bM) = \raisebox{0.05cm}{$\Homeo(\bM)$}\big/\raisebox{-0.05cm}{\footnotesize{homotopy}}$.
The natural map:
\begin{equation*}\label{Eq:laudenbach}
  \MCG(\bM) \to \Out
\end{equation*}
is surjective and has finite kernel generated by Dehn twists about embedded 2--spheres in $\bM$~\cite{bk:Laudenbach74}.  Such homomorphisms act trivially on spheres systems and hence there is a left action of $\Out$ by automorphisms on the sphere graph.  Specifically, realize the given outer automorphism by a homeomorphism of $\bM$ and apply this homeomorphism to the members of a given equivalence class of sphere systems.

\begin{definition}\label{def:free splitting}
A \emph{free splitting} $G$ is a simplicial tree equipped with a cocompact action of $\F_{n}$ by automorphisms (without inversions) such that the stabilizer of every edge is trivial.  We specifically allow for the possibility that a free splitting contains vertices of valence two.  A free splitting is \emph{filling} if the stabilizer of every vertex is trivial. 

We define a preorder on the set of free splittings by $G \preceq \Gamma$ if there is an $\F_{n}$--equivariant cellular map $\Gamma \to G$ with connected point pre-images.  This induces an equivalence relation: $G\sim \Gamma$ if $G \preceq \Gamma$ and $\Gamma \preceq G$.  The set of equivalence classes of free splittings is denoted by $\calX$; the subset of equivalence classes of filling free splittings is denoted $\cv$\footnote{Experts may recognize $\cv$ as the vertices in the spine of the Culler--Vogtmann outer space~\cite{ar:CV86}.}.  When there can be no confusion, we denote the equivalence class of $G$ again by $G$.    

The preorder induces a partial order on $\calX$ that we continue to denote by $\preceq$.  The \emph{free splitting graph} is the simplicial graph with vertex set $\calX$ and edges corresponding to domination $G \preceq \Gamma$.  For $G, \Gamma \in \calX$, we denote by $d_{\calX}(G,\Gamma)$ the distance between $G$ and $\Gamma$ in the free splitting graph.  This the fewest number of edges in an edge path between the two vertices.
\end{definition}

Suppose that $G$ is a free splitting and let $\rho \from \F_{n} \to \Aut(G)$ denote the action homomorphism.  Given $\Phi \in \Aut(\F_{n})$, the homomorphism $\rho \circ \Phi \from \F_{n} \to \Aut(G)$ defines a new free splitting we denote by $G \cdot \Phi$.  This defines a right action by $\Aut(\F_{n})$ on the free splitting graph.  As $\Inn(\F_{n})$ acts trivially, this induces an action of $\Out$ by automorphisms on the free splitting graph.
  
There is a natural $\Out$--equivariant map from the sphere graph to the free splitting graph.  Given a sphere system $S \subset M$, we define a tree $G$ with vertex set consisting of the components of $M - S$ and edges corresponding to non-empty intersection between the closures of the components.  The action of $\F_{n}$ on $M$ induces a cocompact action of $\F_{n}$ on $G$ by automorphisms such that the stabilizer of every edge is trivial, i.e., $G$ is a free splitting.   This map is a simplicial isomorphism~\cite[Lemma~2]{ar:AS11}.  
 
\section{The Guirardel core}\label{sec:core}

In this section we give the definition of Guirardel core of two trees as it is presented in~\cite{ar:Guirardel05} specialized to the case of trees in $\cv$.  

\subsection{A core for a pair of tree actions}\label{subsec:defns} A \emph{ray} in $G \in \cv$ is an isometric embedding $\vec{r} \from \R_+ \to G$.  An \emph{end} of $G$ is an equivalence class of rays under the equivalence relation of having finite Hausdorff distance.  The set of all ends is called the \emph{boundary} of $G$ and is denoted by $\partial G$.  
 
A \emph{direction} is a connected component of $G - \{x\}$, where $x$ is a point in $G$.  A direction $\delta \subset G$ determines a subset $\partial \delta \subset \partial G$ consisting of all ends for which every representative ray intersects $\delta$ in a non-empty (equivalently unbounded) subset.  Given an edge $e \subset G$ we denote by $\vec{e}$ the edge with a specific orientation.  This determines a direction $\delta_{\vec{e}} \subset G$ by taking the component of $G - \{x\}$ that contains $e$, where $x$ is the initial vertex of $\vec{e}$.  We will denote by $ \vec{e}_\infty \subset \partial_\infty G$ the set of ends with a representative that crosses $\vec{e}$ with the specified orientation, i.e., $\vec{e}_\infty = \partial \delta_{\vec{e}}$.

A \emph{quadrant} in $G \times \Gamma$ is the product $\delta_1\times\delta_2$ of two directions $\delta_1\subset G$ and $\delta_2\subset \Gamma$. 

\begin{definition}\label{def:heavy}
Fix a basepoint $(\ast_{1},\ast_{2}) \in G \times \Gamma$ and consider a quadrant $Q=\delta_1\times\delta_2\subset G \times \Gamma$.  We say that $Q$ is \emph{heavy} if there exists a sequence $g_k \in \F_n$ so that 
\begin{enumerate}
\item $(g_k \ast_{1},g_{k}\ast_{2}) \in Q$
\item $d_{G}(g_k \ast_1, \ast_1) \overset{k\rightarrow \infty}{\longrightarrow} \infty$
  and $d_{\Gamma}(g_k \ast_2, \ast_2) \overset{k\rightarrow \infty}{\longrightarrow} \infty$
\end{enumerate}
Otherwise, we say that $Q$ is \emph{light}.
\end{definition}

The core of $G \times \Gamma$ is what remains when one has removed the light quadrants.

\begin{definition}[The Guirardel core]
Suppose that $G,\Gamma \in \cv$ and let $\calL(G,\Gamma)$ be the collection of light quadrants of $G \times \Gamma$.  The \emph{\textup{(}Guirardel\textup{)} core} of $G$ and $\Gamma$ is the subset
\begin{equation*}
\Core(G, \Gamma) = (G \times \Gamma) - 
  \left(\bigcup_{Q \in \calL(G,\Gamma)}  Q \right). 
\end{equation*}
\end{definition}

It follows from the definition that $\Core(G,\Gamma)$ is isomorphic to $\Core(\Gamma,G)$ via the swap $(x,y) \mapsto (y,x)$.  For more details and examples see~\cite{ar:Guirardel05,ar:BBC10}.

\subsection{Computing the core}\label{subsec:compute}

There is an algorithm to compute the core for trees $G,\Gamma \in \cv$.  This suffices to compute the core for any free splittings $G_{0},\Gamma_{0} \in \calX$.  Indeed, if the given trees are not filling, they can be ``blown-up'' to filling trees $G,\Gamma \in \cv$ by replacing vertices with non-trivial stabilizer in the quotient graph of groups $G_{0}/\F_{n}$ and $\Gamma_{0}/\F_{n}$ with roses of of the appropriate rank.  There are domination maps $p \from G \to G_{0}$ and $\pi \from \Gamma \to \Gamma_{0}$ and we have that $\Core(G_{0},\Gamma_{0}) = (p \times \pi)(\Core(G,\Gamma))$.  This material appears in~\cite[Section~2]{ar:BBC10} with slightly different terminology and notation.  We provide proofs of the most relevant parts necessary for the sequel.  

\begin{definition}\label{def:morphism}
Suppose that $G, \Gamma \in \cv$.  An $\F_{n}$--equivariant map $f \from G \to \Gamma$ is called a \emph{morphism} if:
\begin{enumerate} 
\item $f$ linearly expands every edge across a tight edge path; and
\item at each vertex of $G$ there are adjacent edges $e,e'$ such that $f(e) \cap f(e')$ is trivial, i.e., there is more than one gate at each vertex.
\end{enumerate}
\end{definition}
   
Such a map $f$ induces an $\F_{n}$--equivariant homeomorphism $f_\infty \from \partial G \to \partial \Gamma$.  Indeed, this follows by bounded cancellation, for instance see~\cite{ar:Cooper87}.  Next, we state the criterion provided in~\cite{ar:BBC10} regarding the existence of squares in the core.

\begin{lemma}[{\cite[Lemma 2.3]{ar:BBC10}}]\label{lem:squares}
Let $f \from G \to \Gamma$ be a morphism between $G,\Gamma \in \cv$.  Given two edges $e \subset G$ and $\eta \subset \Gamma$, the square $e \times \eta$ is in the core $\Core(G,\Gamma)$ if and only if for every choice of orientations of the edges $e$ and $\eta$ the subset $f_\infty(\vec{e}_\infty) \cap \vec{\eta}_\infty$ is non-empty.
\end{lemma}
  
This condition is very natural in the following way.  Given a curve $\balpha$ on a closed surface $\bX$, each lift $\alpha$ of $\balpha$ to the universal cover $X$ determines a partition of $\partial X$ (which is homeomorphic to $S^{1}$) into two subsets $\alpha_{+}$ and $\alpha_{-}$.  For two curves $\balpha, \bbeta$ on $\bX$ that intersect minimally, lifts $\alpha, \beta$ to $X$ intersect if and only if for every choice of $*,*' \in \{+,-\}$ the set $\alpha_{*} \cap \beta_{*'}$ is non-empty.    
  
Using $f$, it is a simple matter to determine when this condition is met for a given pair of edges.  We discuss this now.  By the \emph{interior} of a simplicial subtree we mean all non-extremal edges. 

\begin{definition}\label{def:hull}
Suppose that $G,\Gamma \in \cv$, $f\from G \to \Gamma$ is a morphism and $\eta \in \Gamma$ is an edge.  We let $\calP^{f}_\eta$ be the set of edges in $G$ whose image under $f$ traverses $\eta$.  In other words, $\calP^{f}_\eta$ is the set of edges containing $f^{-1}(\eta)$.  Since $f$ is a morphism, by bounded cancellation, the set $\calP^{f}_\eta$ is finite.  

Let $\calH^{f}_\eta$ be the interior of the convex hull of $\calP^{f}_\eta$ and let $\widehat{\calP}^{f}_\eta = \calP^{f}_\eta - \calH^{f}_\eta$.  Notice that the interior of the convex hull of $\widehat{\calP}^{f}_{\eta}$ is also $\calH^{f}_{\eta}$.  Suppose $e \in \calH^{f}_\eta$ and $\vec e$ is an orientation of $e$.  We say $\vec e$ \emph{can escape} $\calP^{f}_\eta$ if there is an embedded ray of the form $\vec{e} \cdot \vec{r}$ such that $\vec{r}$ does not cross any edge of $\widehat{\calP}^{f}_\eta$. Define the \emph{consolidated convex hull} $\CH^{f}_\eta$ of $\calP^{f}_\eta$ to be the set of edges in $e \in \calH^{f}_\eta$ such that both orientations of $e$ can escape $\calP^{f}_\eta$. 
\end{definition}

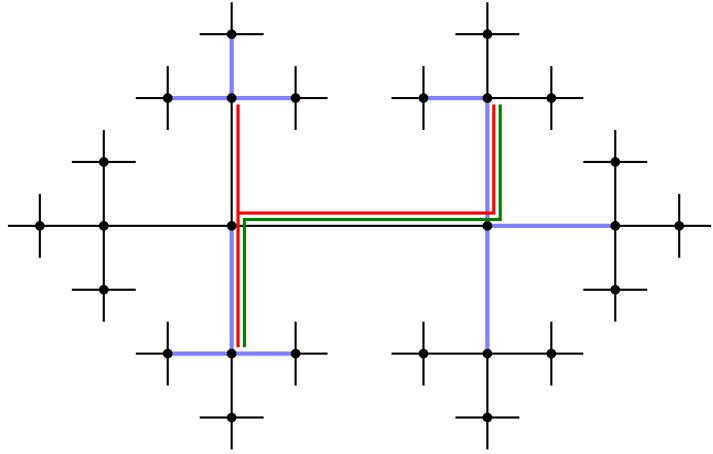
\begin{figure}[ht]
\begin{tikzpicture}[scale=1.7]
\tikzstyle{vertex} =[circle,draw,fill=black,thick, inner sep=0pt,minimum size= 1 mm]
\draw[thick] (-2.75,0) -- (2.75,0);
\draw[thick] (-1,-1.75) -- (-1,1.75);
\draw[thick] (1,-1.75) -- (1,1.75);
\draw[thick] (-2,-0.75) -- (-2,0.75);
\draw[thick] (2,-0.75) -- (2,0.75);
\draw[thick] (0.25,1) -- (1.75,1);
\draw[thick] (-0.25,1) -- (-1.75,1);
\draw[thick] (0.25,-1) -- (1.75,-1);
\draw[thick] (-0.25,-1) -- (-1.75,-1);
\draw[thick] (2.5,0.25) -- (2.5,-0.25);
\draw[thick] (-2.5,0.25) -- (-2.5,-0.25);
\draw[thick] (1.75,0.5) -- (2.25,0.5);
\draw[thick] (-1.75,0.5) -- (-2.25,0.5);
\draw[thick] (1.75,-0.5) -- (2.25,-0.5);
\draw[thick] (-1.75,-0.5) -- (-2.25,-0.5);
\draw[thick] (0.75,1.5) -- (1.25,1.5);
\draw[thick] (-0.75,1.5) -- (-1.25,1.5);
\draw[thick] (0.75,-1.5) -- (1.25,-1.5);
\draw[thick] (-0.75,-1.5) -- (-1.25,-1.5);
\draw[thick] (1.5,0.75) -- (1.5,1.25);
\draw[thick] (-1.5,0.75) -- (-1.5,1.25);
\draw[thick] (1.5,-0.75) -- (1.5,-1.25);
\draw[thick] (-1.5,-0.75) -- (-1.5,-1.25);
\draw[thick] (0.5,0.75) -- (0.5,1.25);
\draw[thick] (-0.5,0.75) -- (-0.5,1.25);
\draw[thick] (0.5,-0.75) -- (0.5,-1.25);
\draw[thick] (-0.5,-0.75) -- (-0.5,-1.25);
\draw[ultra thick,LightBlue] (1,0) -- (2,0);
\draw[ultra thick,LightBlue] (1,1) -- (1,-1);
\draw[ultra thick,LightBlue] (0.5,1) -- (1,1);
\draw[ultra thick,LightBlue] (-1,1) -- (-1,1.5);
\draw[ultra thick,LightBlue] (-1.5,1) -- (-0.5,1);
\draw[ultra thick,LightBlue] (-1,0) -- (-1,-1);
\draw[ultra thick,LightBlue] (-1.5,-1) -- (-0.5,-1);
\draw[very thick,red] (-0.95,0.95) -- (-0.95,-0.95);
\draw[very thick,red] (-0.95,0.1) -- (1.05,0.1) -- (1.05,0.95);
\draw[very thick,green!50!black] (-0.9,-0.95) -- (-0.9,0.05) -- (1.1,0.05) -- (1.1,0.95);
\node [vertex] at (-1,0) {};
\node [vertex] at (1,0) {};
\node [vertex] at (2,0) {};
\node [vertex] at (2.5,0) {};
\node [vertex] at (-2,0) {};
\node [vertex] at (-2.5,0) {};
\node [vertex] at (-2,0.5) {};
\node [vertex] at (2,0.5) {};
\node [vertex] at (2,-0.5) {};
\node [vertex] at (-2,-0.5) {};
\node [vertex] at (-1.5,1) {};
\node [vertex] at (-1,1) {};
\node [vertex] at (-0.5,1) {};
\node [vertex] at (0.5,1) {};
\node [vertex] at (1,1) {};
\node [vertex] at (1.5,1) {};
\node [vertex] at (-1.5,-1) {};
\node [vertex] at (-1,-1) {};
\node [vertex] at (-0.5,-1) {};
\node [vertex] at (0.5,-1) {};
\node [vertex] at (1,-1) {};
\node [vertex] at (1.5,-1) {};
\node [vertex] at (-1,1.5) {};
\node [vertex] at (-1,-1.5) {};
\node [vertex] at (1,1.5) {};
\node [vertex] at (1,-1.5) {};
\end{tikzpicture}
\caption{A schematic of the sets $\calP^{f}_\eta$ (blue), $\widehat{\calP}^{f}_\eta$, $\calH^{f}_\eta$ (red) and $\CH^{f}_\eta$ (green).}
\end{figure}

\begin{lemma}\label{lem:vertices escape}
For every vertex $v \in G$ there is a ray $\vec r$ originating at $v$ that is disjoint from $\calP_\eta^{f}$. 
\end{lemma}

\begin{proof}
If the lemma were false, then for every edge $e$ adjacent to $v$ the image $f(e)$ would contain the initial edge in the path connecting $f(v)$ to $\eta$.  This violates condition (2) in Definition~\ref{def:morphism}.
\end{proof}

The following simple condition tells exactly when a square is in the core.

\begin{lemma}\label{lem:hull}
Let $G,\Gamma \in \cv$, fix morphisms $f \from G \to \Gamma$ and $\phi \from \Gamma \to G$ and consider a pair of edges $e \subset G$ and $\eta \subset \Gamma$.  The square $e \times \eta$ is in $\Core(G,\Gamma)$ if and only if one of the two following equivalent conditions holds:
\begin{itemize}
\item $e \subseteq \CH_\eta^{f}$. 
\item  $\eta \subseteq \CH_e^{\phi}$. 
\end{itemize}
\end{lemma}

\begin{proof}
We prove the first of the two equivalent statements; the fact that they are equivalent follows from the symmetry of the construction of the core.  For simplicity, we omit the superscript $f$ on the various subsets from Definition~\ref{def:hull} during the proof of this lemma.

By Lemma~\ref{lem:squares}, what needs to be shown is that $e \subseteq \CH_\eta$ if and only if for each orientation $\vec{e}$ of $e$ and orientation $\vec{\eta}$ of $\eta$ there is a ray $\vec{r}$ crossing $\vec{e}$ with the specified orientation such that $f_\infty(\vec{r}\,) \in \vec{\eta}_\infty$.  

First suppose that $e \subseteq \CH_\eta$ and fix an orientation $\vec{e}$ on $e$. 
As $e \subseteq \CH_\eta$, there is a ray $\vec{r}_0 = \vec{e} \cdot \vec{u}$ such that $\vec{u}$ is disjoint from $\widehat{\calP}_\eta$.  Let $e_0$ be the last edge on $\vec{r}_0$ that is in $\CH_\eta$ and decompose $\vec{r}_0 = \vec{u}_0 \cdot \vec{e}_0 \cdot \vec{u}_1$ where $\vec{u}_0$ may be trivial.  It is easy to verify that the ray $\vec{u}_1$ is disjoint from $\calP_\eta$.  As $e_0 \subseteq \calH_\eta$, there is a ray of the form $\vec{e}_0 \cdot \vec{u}_2$ where $\vec{u}_2$ is not disjoint from $\calP_\eta$.  (It may be that $\vec{u}_{1}$ and $\vec{u}_{2}$ have non-trivial intersection.)  Let $e_1$ be the first edge on $\vec{u}_2$ that is contained in $\calP_\eta$ and $\vec{p}$ the oriented edge path from $\vec{e}$ to $\vec{e}_1$.  By Lemma~\ref{lem:vertices escape}, there is a ray $\vec{v}_1$ originating at the terminal vertex of $\vec{p}$ that is disjoint from $\calP_\eta$.  Let $\vec{r}_1 = \vec{p} \cdot \vec{v}_1$.  We now see that:
\begin{equation*}
\# | \vec{r}_1 \cap \calP_\eta | = \# | \vec{r}_0 \cap \calP_\eta| + 1
\end{equation*}
Since $\vec{r}_0$ and $\vec{r}_1$ originate from the same vertex, their $f_\infty$--images lie in $\vec{\eta}_\infty$ for opposite choices of orientation of $\eta$.   By Lemma~\ref{lem:squares}, this shows that $e  \times \eta \subseteq \Core(G,\Gamma)$.

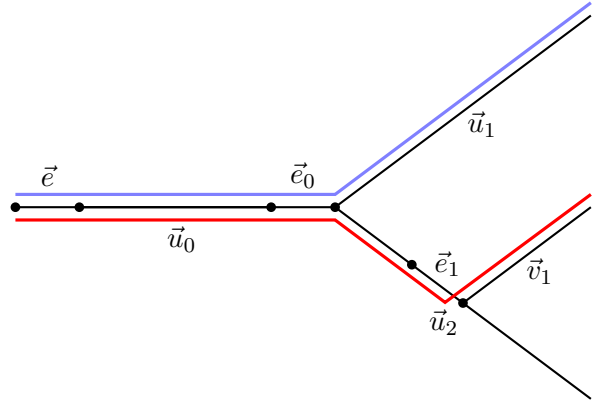
\begin{figure}[ht]
\begin{tikzpicture}[scale=1.7]
\tikzstyle{vertex} =[circle,draw,fill=black,thick, inner sep=0pt,minimum size= 1 mm]
\draw[thick] (-2,0) -- (0.5,0);
\draw[thick] (0.5,0) -- (2.5,-1.5);
\draw[thick] (1.5,-0.75) -- (2.5,0);
\draw[thick] (-1.5,0) -- (0,0);
\draw[thick] (0.5,0) -- (2.5,1.5);
\draw[very thick,LightBlue] (-2,0.1) -- (0.5,0.1) -- (2.5,1.6);
\draw[very thick,red] (-2,-0.1) -- (0.5,-0.1) -- (1.36,-0.745) -- (2.5,0.1);
\node[vertex] at (-2,0) {};
\node[vertex] at (-1.5,0) {};
\node[vertex] at (0,0) {};
\node[vertex] at (0.5,0) {};
\node[vertex] at (1.5,-0.75) {};
\node[vertex] at (1.1,-0.45) {};
\node at (-1.75,0.25) {$\vec{e}$};
\node at (0.25,0.25) {$\vec{e}_{0}$};
\node at (1.38,-0.44) {$\vec{e}_{1}$};
\node at (-0.7,-0.25) {$\vec{u}_{0}$};
\node at (1.65,0.65) {$\vec{u}_{1}$};
\node at (1.35,-0.9) {$\vec{u}_{2}$};
\node at (2.1,-0.5) {$\vec{v}_{1}$};
\end{tikzpicture}\caption{Rays $\vec{r}_{0}$ (blue) and $\vec{r}_{1}$ (red) witnessing $e \subseteq \CH_\eta$ in Lemma~\ref{lem:hull}.}
\end{figure}

For the converse we suppose that $e \nsubseteq \CH_\eta$.  If further, $e \nsubseteq \calH_\eta$, then there is a choice of orientation $\vec{e}$ such that for every ray of the form $\vec{e} \cdot \vec{r}$, the ray $\vec{r}$ misses $\calP_\eta$.  Therefore, there is an orientation on $\eta$, say $\vec{\eta}$, such that $f_\infty(\vec{e}_\infty) \cap \vec{\eta}_\infty = \emptyset$.  By Lemma~\ref{lem:squares}, $e \times \eta \nsubseteq \Core(G,\Gamma)$.

Thus we can assume that $e \subseteq \calH_\eta - \CH_\eta$.  Hence, there is a choice of orientation $\vec{e}$ that cannot escape, i.e., for every ray form $\vec{e} \cdot \vec{r}$, the ray $\vec{r}$ must contain some edge in $\widehat{\calP}_\eta$. By Lemma~\ref{lem:vertices escape}, we see that each such ray $\vec{r}$ can only contain a single edge of $\calP_\eta$.  Again, there is an orientation on $\eta$, say $\vec{\eta}$, such that $f_\infty(\vec{e}_\infty) \cap \vec{\eta}_\infty = \emptyset$.  By Lemma~\ref{lem:squares}, $e \times \eta \nsubseteq \Core(G,\Gamma)$.  
\end{proof}

Since the $\Core(G,\Gamma)$ is defined without reference to the morphism $f \from G \to \Gamma$, Lemma~\ref{lem:hull} shows that $\CH_{\eta}^{f}$ and $\CH_{e}^{\phi}$ do not depend on the actual morphism used to compute them.  As such, we will drop the superscripts from these sets for the remainder.

%% file: sphere.tex

\section{Sphere systems and the core}\label{sec:sphere-core}

In Section~\ref{sec:sphere-splittings} we described an $\Out$--equivariant association between sphere systems and free splittings respecting the notion of filling: $(\calS, \sphere) \leftrightarrow (\calX,\cv)$.  In Section~\ref{sec:core}, given a pair of free splittings $G,\Gamma \in \cv$, we described how to construct their Guirardel core $\Core(G, \Gamma)$.  The goal of this section is given a pair of filling sphere systems $S, \Sigma \in \sphere$ to construct a $2$--dimensional square complex $\Core(S, \Sigma)$.   We then show that when the pair of sphere systems $(S,\Sigma)$ is associated to the pair of free splittings $(G,\Gamma)$ there is a $\F_{n}$--equivariant isomorphism from $\Core(S, \Sigma) \to \Core(G, \Gamma)$ of square complexes.  Moreover, this association is $\Out$--equivariant with respect to the actions on $\cv$ and $\sphere$.  This association is implicit in the proof of Proposition~2.1 in~\cite{ar:Horbez12}.  We explain the connection in more details here and provide an alternative proof.  The in depth description is necessary for understanding the effect of surgery on the core that we describe in Section~\ref{sec:surgery and core}.

\subsection{Hatcher's normal form}\label{subsec:normal form}

Central to the understanding of sphere systems in $M$ is Hatcher's notion of normal form.  He originally defined normal form only with respect to a maximal sphere systems $\Sigma$~\cite{ar:Hatcher95} and extended this to filling sphere systems in subsequent work with Vogtmann~\cite{ar:HV96}.  We recall this definition now.  The sphere system $S$ is said to be in \emph{normal form} with respect to $\Sigma$ if every sphere $s \in S$ either belongs to $\Sigma$, or intersects $\Sigma$ transversely in a collection of circles that split $s$ into components called \emph{pieces} such that for each component $\Pi \subset M - \Sigma$ one has:
\begin{enumerate}
\item each piece in $\Pi$ meets each boundary sphere in $\partial \Pi$ in at most one circle, and
\item no piece in $\Pi$ is a disk that is isotopic relative to its boundary to a disk in $\partial \Pi$.
\end{enumerate}

Hatcher proved that a sphere system $S$ can always be homotoped into
normal form with respect to the maximal sphere system $\Sigma$ and that such a form is unique up to homotopy~\cite{ar:Hatcher95,ar:HV96}.  Hensel--Osajda--Przytycki generalized Hatcher's definition of normal form to non-filling sphere systems and in a way that is obviously symmetric with respect to the two sphere systems~\cite{ar:HOP14}.   With their notion, two sphere systems $S$ and $\Sigma$ are in \emph{normal form} if for all $s \in S$ and $\sigma \in \Sigma$ one has:  
\begin{enumerate}
\item $s$ and $\sigma$ intersect transversely in at most one circle, and
\item none of the disks in $s - \sigma$ is isotopic relative to its boundary to a disk in $\sigma$.
\end{enumerate}
These notions are equivalent when $\Sigma$ is filling~\cite[Section~7.1]{ar:HOP14}.  

\subsection{A core for a pair of sphere systems}\label{subsec:sphere-core}
Suppose that $S$ and $\Sigma$ are filling sphere systems in $M$ and that they are in normal form.  An \emph{$S$--piece} is the closure of a component of $S - \Sigma$.  Likewise, a \emph{$\Sigma$--piece} is the closure of a component of $\Sigma - S$.  By \emph{piece}, we mean either an $S$--piece or a $\Sigma$--piece (this agrees with the use of piece in Section~\ref{subsec:normal form}).

\begin{lemma}\label{lem:intersection of components}
Suppose that $X$ is the intersection of a component of $M - S$ and a component of $M - \Sigma$.  Then the following statements are true.
\begin{enumerate}

\item $X$ is connected.

\item $\partial X$ is the union of $S$--pieces and $\Sigma$--pieces and moreover, different pieces are subsets of different spheres.

\item If $Y$ is also the intersection of a component of $M - S$ and a component of $M - \Sigma$, then either $X = Y$, their closures are disjoint, or $\partial X \cap \partial Y$ is a piece. 

\end{enumerate}
\end{lemma}

\begin{proof}
This follows from the description of normal form, details are left to the reader.
\end{proof}

The first item in Lemma~\ref{lem:intersection of components} implies that the intersection of a component of $M - S$ and a component of $M - \Sigma$ is either empty or a component of $M - (S \cup \Sigma)$.

\begin{definition}\label{def:sphere core}
Suppose that $S$ and $\Sigma$ are filling sphere systems in $M$ and that they are normal form. The \emph{core} of $S$ and $\Sigma$, denoted $\Core(S,\Sigma)$, is the square complex defined as follows. 

\begin{itemize}

\item Vertices correspond to components of $M - (S \cup \Sigma)$.  Such a region corresponds to the intersection of a component $P \subset M - S$ and a component $\Pi \subset M - \Sigma$.  We denote the vertex by $(P,\Pi)$.

\item There is an edge between two vertices when the closures of the corresponding components of $M - (S \cup \Sigma)$ have non-trivial intersection.  By Lemma~\ref{lem:intersection of components}, each edge corresponds to a piece.  If it is an $S$--piece, then it is the closure of $s \cap \Pi$ for some sphere $s \in S$ and component $\Pi \subset M - \Sigma$.  We denote the edge by $(s, \Pi)$.  Likewise, if it is an $\Sigma$--piece, then it is the closure of $P \cap \sigma$ for some component $P \subset M - S$ and sphere $\sigma \in \Sigma$ .  In this case, we denote the edge by $(P,\sigma)$. 

\item Suppose that $s \in S$ and $\sigma \in \Sigma$ have non-empty intersection.  Let $P_1$, $P_2 \subset M - S$ be the components whose boundary contains $s$ and let $\Pi_1$, $\Pi_2 \subset M - \Sigma$ be the components whose boundary contains $\sigma$.  Then four edges $(s,\Pi_1)$, $(P_1,\sigma)$, $(s,\Pi_2)$ and $(P_2,\sigma)$ form the boundary of a square with vertices $(P_{1},\Pi_{1})$, $(P_{2},\Pi_{1})$, $(P_{2},\Pi_{2})$ and $(P_{1},\Pi_{2})$ which is then filled in.  The square is denoted by $s \times \sigma$.

\end{itemize}
\end{definition}

\begin{remark}\label{rem:diagonal edge}
We always assume that $S$ and $\Sigma$ do not share a sphere.  Otherwise Theorem~\ref{Thm:Main} is trivial.  However the core in this case would be disconnected and make the exposition more complicated.  There is a procedure to add diagonal edges resulting in the \emph{augmented core}, which is connected.  See~\cite{ar:Guirardel05} for details.  
\end{remark}

\begin{figure}[ht]
\begin{tikzpicture}[scale=0.6]
\tikzstyle{vertex} =[circle,draw,fill=black,thick, inner sep=0pt,minimum size= 1 mm]
\filldraw[yellow,opacity=.7] (0,0,0) -- (-2.5,0,0) -- (-2.5,-2.5,0) -- (0,-2.5,0) -- cycle;
\draw[thick] (-2.5,0,0) -- (-2.5,-2.5,0) -- (0,-2.5,0);
\filldraw[green!30,opacity=.5] (0,0,5) -- (0,-5,5) -- (0,-5,-5) -- (0,0,-5) -- cycle;
\filldraw[yellow,opacity=.7] (0,0,0) -- (2.5,0,0) -- (2.5,-2.5,0) -- (0,-2.5,0) -- cycle;
\draw[thick] (2.5,0,0) -- (2.5,-2.5,0) -- (0,-2.5,0);
\filldraw[red!30,opacity=.5] (-5,0,5) -- (5,0,5) -- (5,0,-5) -- (-5,0,-5) -- cycle;
\filldraw[yellow,opacity=.7] (0,0,0) -- (-2.5,0,0) -- (-2.5,2.5,0) -- (0,2.5,0) -- cycle;
\draw[thick] (-2.5,0,0) -- (-2.5,2.5,0) -- (0,2.5,0);
\filldraw[green!30,opacity=.5] (0,0,5) -- (0,5,5) -- (0,5,-5) -- (0,0,-5) -- cycle;
\draw[thick] (0,0,5) -- (0,0,-5);
\filldraw[yellow,opacity=.7] (0,0,0) -- (2.5,0,0) -- (2.5,2.5,0) -- (0,2.5,0) -- cycle;
\draw[thick] (2.5,0,0) -- (2.5,2.5,0) -- (0,2.5,0);
\node[vertex] at (2.5,2.5,0) {};
\node[vertex] at (-2.5,2.5,0) {};
\node[vertex] at (-2.5,-2.5,0) {};
\node[vertex] at (2.5,-2.5,0) {};
\node at (2.9,2.95,0) {$(P_1,\Pi_1)$};
\node at (-2.9,2.95,0) {$(P_2,\Pi_1)$};
\node at (2.9,-2.95,0) {$(P_1,\Pi_2)$};
\node at (-2.9,-2.95,0) {$(P_2,\Pi_2)$};
\node at (0,2.95,0) {$(s,\Pi_1)$};
\node at (0,-2.95,0) {$(s,\Pi_2)$};
\node at (3.5,0,0) {$(P_1,\sigma)$};
\node at (-3.5,0,0) {$(P_2,\sigma)$};
\node at (5,1,0) {$\sigma$};
\node at (1,5,0) {$s$};
\end{tikzpicture}
\caption{Edges $(s,\Pi_1)$, $(P_1,\sigma)$, $(s,\Pi_2)$ and $(P_2,\sigma)$ form the boundary of a square $s \times \sigma$.}
\end{figure}
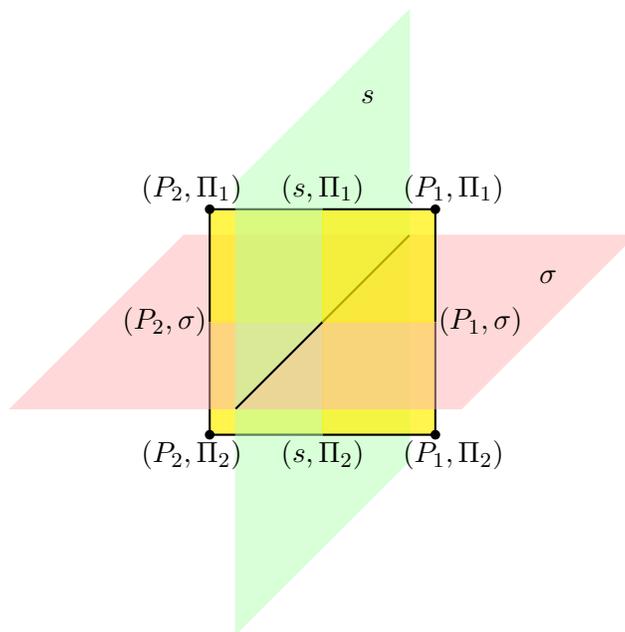

Let $G$, $\Gamma \in \cv$ be the free splittings corresponding to $S$ and $\Sigma$ respectively.  We will show that the two notions of the core, $\Core(S,\Sigma)$ and $\Core(G,\Gamma)$, are isomorphic as $\F_{n}$--square complexes.  We will do so by showing that their horizontal hyperplanes agree.  To this end we make the following definition.

\begin{definition}\label{def:shadow}
The \emph{shadow of $\sigma \in \Sigma$} is the union of the edges $e \subset G$ whose associated sphere in $S$ intersects $\sigma$.  We denote the shadow by $\Shadow(\sigma) \subset G$. 
\end{definition}

Observe that the shadow of $\sigma$ is isomorphic to the tree in $\sigma$ that is dual to the intersection circles between $\sigma$ and $S$.  Now will show how to relate the two definitions of the core.  We will make use of the following notion.  

\begin{definition}\label{def:dual graph}
If $G \in \calX$ corresponds to a sphere system $S \in \calS$, and $\iota \from G \hookrightarrow M$ is an $\F_{n}$--equivariant embedding, we say $\iota(G)$ is \emph{dual to $S$} if each sphere $s \in S$ intersects exactly one edge of $\iota(G)$, namely the image of the corresponding edge, and this intersection is transverse and a single point.  We say that $\iota$ is a \emph{dual embedding \textup{(}for $S$\textup{)}}.
\end{definition}

It is a routine matter to construct a dual embedding for a given free splitting.  We need to show that we can make it in some sense normal to $\Sigma$.

\begin{lemma}\label{lem:normal embedding}
There exists a dual embedding $\iota \from G \hookrightarrow M$ for $S$ so that for each edge $e \in G$ and sphere $\sigma \in \Sigma$, $\iota(e)$ and $\sigma$ are either disjoint or intersect transversely at a single point in the interior of $\iota(e)$.   
\end{lemma}

\begin{proof}
Let $\iota_{0} \from G \hookrightarrow M$ be a dual embedding.  By general position, we can assume that $\iota_{0}(G) \cap S \cap \Sigma = \emptyset$ and that $\Sigma$ is disjoint from the vertices of $\iota_{0}(G)$.

Suppose that for some edge $e \subset G$, the image $\iota_{0}(e)$ intersects some sphere in $\Sigma$ in more than one point.  Let $s \in S$ be the sphere corresponding to $e$.  Fix some innermost pair of intersection points $x,y \in \iota_{0}(e)$ and let $\sigma \in \Sigma$ be the corresponding sphere.  Let $I$ be the subsegment of $\iota_{0}(e)$ with endpoints $x$ and $y$.  

Notice that any circle of intersection of $S \cap \sigma$ that separates $x$ and $y$ in $\sigma$ must correspond to a sphere $s' \in S$ such that $s' \cap I \neq \emptyset$.  Indeed, if not, since $S$ and $\Sigma$ are in normal form, there would a loop consisting of $I$ and an arc in $\sigma$ that intersects some sphere in $S$ exactly once.  This is a contradiction as spheres in $S$ are separating.  

Therefore, there is an arc $J \subset \sigma$ that intersects exactly the same set of spheres of $S$ as $I$, which is either $s$ or the empty set.  We can then homotope $I$ to $J$ and continue pushing in this direction to reduce the number of intersection points between $\iota_{0}(e)$ and $\Sigma$ by two.  Equivariantly perform this process to obtain a new dual embedding $\iota_{1} \from G \to M$ that has fewer $F_{n}$--orbits of intersect between  the image of $G$ and $\Sigma$.

Iterating this procedure we arrive at $\iota \from G \to M$ as in the statement of the lemma.
\end{proof}

If $\iota \from G \hookrightarrow M$ is an $\F_{n}$--equivariant embedding so that $\iota(G)$ is transverse to $\Sigma$ we can create a map $k_{\iota} \from G \to \Gamma$ by sending an edge $e$ to the edge path in $\Gamma$ corresponding to the spheres in $\Sigma$ crossed by $\iota(e)$.

\begin{lemma}\label{lem:morphism}
There exists a dual embedding $\iota \from G \hookrightarrow M$ for $S$ so that the associated map $k_{\iota} \from G \to \Gamma$ is a morphism.     
\end{lemma}

\begin{proof}
Whenever a dual embedding $\iota_{0} \from G \hookrightarrow M$ satisfies Lemma~\ref{lem:normal embedding} the image of each edge $e \subset G$ is a tight edge path in $\Gamma$.  Thus the only way that such a dual embedding $\iota_{0}\from G \to M$ fails to produce a morphism is if there is some vertex $v \in G$ with adjacent edges $e_{1},\ldots, e_{\ell}$ (oriented to have $v$ as their initial vertex) and sphere $\sigma \in \Sigma$ so that the first intersection point of $\iota_{0}(e_{i}) \cap \Sigma$ lies in $\sigma$ for each $i = 1,\ldots, \ell$.  Arguing as in Lemma~\ref{lem:normal embedding} we can equivariantly homotope $\iota_{0}$ to locally reduce the number of intersections between the image of $G$ and $\Sigma$ by pushing the image of $v$ across $\sigma$ and pushing subarcs of edges with both endpoints on $\sigma$ across $\sigma$ as well.  

Iterating this procedure we arrive at $\iota \from G \to M$ as in the statement of the lemma.  
\end{proof}

A dual embedding $\iota \from G \to M$ satisfying the conclusions of of Lemma~\ref{lem:normal embedding} and \ref{lem:morphism} is said to be \emph{normal to $\Sigma$}.  

\begin{proposition}\label{prop:shadow = hull}
Suppose that $S$ and $\Sigma$ are filling sphere systems in $M$ and that $G$ and $\Gamma$ are the associated trees.  Fix an edge $\eta \in \Gamma$ and let $\sigma \in \Sigma$ be the associated sphere.  If $\iota \from G \hookrightarrow M$ is a dual embedding that is normal to $\Sigma$ then $\Shadow(\sigma) = \CH_\eta = \calH_{\eta}^{k_{\iota}}$.
\end{proposition}

\begin{proof}
Suppose that $\iota \from G \hookrightarrow M$ is a dual embedding that is normal to $\Sigma$.  Let $e \subset G$ be an edge and $s \in S$ the sphere corresponding to $e$.  

First suppose that $e \subseteq \Shadow(\sigma)$.  Thus $s \cap \sigma$ is non-empty and as the sphere systems are in normal form, this intersection is a single circle.  Let $X$ be one of the four components of $M - (s \cup \sigma)$.  Decompose $\partial X = d \cup \delta$ where $d$ is a subdisk of $s$ and $\delta$ is a subdisk of $\sigma$.

We claim that $\iota^{-1}(X) \subseteq G$ contains an infinite subtree.  Suppose otherwise, thus $\iota^{-1}(X)$ is a finite sub-forest $T$.  At most one extremal vertex of $T$ corresponds to an intersection of $\iota(e)$ and $s$ (which is in $d$), the remaining extremal vertices correspond to intersections of $\delta$ with edges in $\iota(G)$.  

If $T$ is empty or has some component contained in an edge of $G$ then an innermost disk of $\delta$ (with respect to the intersection circles $\delta \cap S$) is homotopic relative to its boundary to a disk in $S$, violating the assumption that $S$ and $\Sigma$ are in normal form.  Else, if for some component $T_{0} \subseteq T$, we have that $\iota(T_{0})$ does not intersect $s$, then for any interior vertex of $T$, as we saw in the proof of Lemma~\ref{lem:morphism}, the map $k_{\iota}$ only has one gate, violating the assumption that $\iota$ is not normal.  

Thus we may assume that $T$ is connected and has some interior vertex $v$, that we assume is adjacent to some extremal edge of $T$ that is not contained in $e$.  We label the edges $e_{0},e_{1},\ldots, e_{\ell}$ adjacent to $v$ where $\iota(e_{i})$ intersects $\sigma$ for $i = 1,\ldots \ell$.  Let $s_{i}$ be the spheres of $S$ corresponding to $e_{i}$ for $i = 0,\ldots \ell$.  Then $\sigma$ must be disjoint from $s_{i}$ for $i = 1,\ldots,\ell$ for otherwise there is a component of $M - (s_{i} \cup \sigma)$ whose pre-image in $G$ contains a component that contained in a single edge, which we already ruled out.  But in this case we have that $\sigma - s_{0}$ contains a disk isotopic relative to its boundary to a disk in $s_{0}$, which again violates the assumption that $S$ and $\Sigma$ are in normal form.     

Hence we can find a ray $\vec r$ starting with $e$ so that $\iota(\vec r)$ is eventually contained in $X$.  Since $X$ was arbitrary, this shows that for each orientation of $\vec e$ for $e$ and $\vec{eta}$ for $\eta$ we can find a ray $\vec r$, crossing $\vec e$ with the specified orientation so that $k_{\iota}(\vec r \, ) \in \vec \eta_{\infty}$.  By Lemma~\ref{lem:squares}, we have that $e \times \eta \subseteq \Core(G,\Gamma)$ and so $e \subseteq \CH_{\eta}$ by Lemma~\ref{lem:hull}.  Hence $\Shadow(\sigma) \subseteq \CH_\eta \subseteq \calH_{\eta}^{k_{\iota}}$.

Now suppose that $e \subseteq \calH^{k_{\iota}}_{\eta}$.  Then for each orientation $\vec e$, there is a ray of the form $\vec e \cdot \vec r$ so that $k_{\iota}(\vec r)$ intersects $\eta$ and hence $\iota(\vec r)$ intersects $\sigma$.  Since $s$ separates $M$ and $\sigma$ is connected, this shows that $s$ intersects $\sigma$, i.e., $e \subseteq \Shadow(\sigma)$.  Hence $\calH^{k_{\iota}}_{\eta} \subseteq \Shadow(\sigma)$ completing the proof.
\end{proof}

In other words, Proposition~\ref{prop:shadow = hull} states that $\Shadow(\sigma) \subset G$ is the interior of the convex hull of $k_{\iota}^{-1}(\sigma)$.

Recall the relation between a sphere system $S$ and the corresponding free splitting $G$ mentioned in Section~\ref{sec:sphere-splittings}: vertices of $G$ correspond to connected components $M - S$ and edges corresponding to non-empty intersection between the closures of the components, i.e., spheres in $S$.  We can define a map $\Core(S,\Sigma) \to G \times \Gamma$ as follows:
\begin{itemize}
\item The image of a vertex $(P,\Pi)$ is the vertex $(v,\nu) \in G \times \Gamma$ where $v$ is the vertex corresponding to $P \subset M - S$ and $\nu$ is the vertex corresponding to $\Pi \subset M - \Sigma$.
\item The image of an edge $(s,\Pi)$ is the edge $(e, \nu) \subset G \times \Gamma$ where $e$ is the edge corresponding to $s \in S$ and $\nu$ is the vertex corresponding to $\Pi \subset M - \Sigma$.  Likewise, the image of an edge $(P,\sigma)$ is the edge $(v, \eta) \subset G \times \Gamma$ where $v$ is the vertex corresponding to $P \subset M - S$  and $\eta$ is the edge corresponding to $\sigma \in \Sigma$.
\item The image of the square $s \times \sigma$ is $e \times \eta \subset G \times \Gamma$ where $e$ is the edge corresponding to $s \in S$ and 
$\eta$ is the edge corresponding to $\sigma \in \Sigma$.
\end{itemize}

The following theorem is implicit in the proof of \cite[Proposition~2.1]{ar:Horbez12}.  There, Horbez uses a characterization by Guirardel of the core as the minimal closed, connect, $\F_{n}$--invariant subset of $G \times \Gamma$ that have connected fibers~\cite[Proposition~5.1]{ar:Guirardel05}.  We avoid using this characterization by using Lemma~\ref{lem:hull} and Proposition~\ref{prop:shadow = hull}.  

\begin{theorem}\label{th:core to core}
If $G,\Gamma \in \cv$ correspond to $S,\Sigma \in \sphere$, which do not share a sphere, then the map $\Core(S,\Sigma) \to G \times \Gamma$ induces an $\F_{n}$--equivariant isomorphism of square complexes $\Core(S,\Sigma) \to \Core(G,\Gamma)$.
\end{theorem}

\begin{proof}
It is clear that the map is injective, $\F_{n}$--equivariant and preserves the square structure.  We just need to show that the image is $\Core(G,\Gamma)$.  For each $\sigma \in \Sigma$, let  $S_{\sigma} = \{ s \in S \mid s \cap \sigma \neq \emptyset\}$.  Notice that the edges in $G$ corresponding to $S_{\sigma}$ is $\Shadow(\sigma)$ by definition.  We can decompose the core $\Core(S,\Sigma)$ vertically into horizontal slices $C_{\sigma} = \{ s \times \sigma \mid s \in S_{\sigma}\}$.  Now fix an $\sigma$ and let $\eta$ by the corresponding edges of $\Gamma$.  Then image of the strip $C_{\sigma}$ is exactly the set of squares $\{ e \times \eta \mid e \subseteq \Shadow(\sigma)\}$.  By Proposition~\ref{prop:shadow = hull} we can also write this as $\{ e \times \eta \mid e \subseteq \CH_{\eta}\}$.  By Lemma~\ref{lem:hull} we can further write this as $\{ e \times \eta \mid e \times \eta \subseteq \Core(G,\Gamma)\}$.  Hence the image of the map is as claimed.  
\end{proof}

%% file: surgery.tex

\section{Surgery and the core}\label{sec:surgery and core}
%
%

The purpose of this section is to show how the core changes along a surgery path in the sphere graph.

\subsection{Surgery Sequences}\label{subsection:surgery}

Suppose that $S, \Sigma \in \calS$ and assume that they are in normal form.  We now describe a path from $S$ to $\Sigma$ in $\calS$ using a surgery procedure introduced by Hatcher~\cite{ar:Hatcher95}.  It is exactly these paths that appear in the the main theorem of this paper. 

Fix a sphere $\sigma \in \Sigma$ that intersects some spheres of $S$.  The intersection circles define a pattern of disjoint circles on $\sigma$, each of which bounds two disks on $\sigma$.  Choose an innermost disk $\delta$ in this collection, i.e., a disk that contains no other disk from this collection, and let $\alpha$ be its boundary circle.  The sphere $s \in S$ containing $\alpha$ is the union of two disks $d_{+}$ and $d_{-}$ that share the boundary circle $\alpha$.  Briefly, surgery replaces the sphere $\sigma$ with new spheres $d_{+} \cup \delta$ and $d_{-} \cup \delta$.  One problem that arises is that the new sphere system and $S$ are not in normal form.  This happens when some innermost disk $\delta'$ in a sphere $\sigma' \in \Sigma$ is parallel rel $s$ to $\delta$.  To address this, we remove all such disks at once so that the resulting sphere system and $S$ are in normal form (Lemma~\ref{lem:normal after surgery}).

Let $\{ \alpha_{i} \}_{i = k}^{\ell}$ be the maximal family of intersection circles in $s \cap \Sigma$ such that:
\begin{enumerate}

\item $k \leq 0 \leq \ell$,

\item $\alpha_{i} \subset d_{-}$ for $i \leq 0$ and $\alpha_{i} \subset d_{+}$ for $i \geq 0$ (this implies that $\alpha_{0} = \alpha$), and

\item for $k \leq i < \ell$, the circles $\alpha_{i}$ and $\alpha_{i+1}$ co-bound an annulus $A_{i} \subset s$ whose interior is disjoint from $\Sigma$.

\end{enumerate}
Related to these circles, we let $\{\delta_{i} \}_{i = \kappa}^{\lambda}$ be the maximal family of innermost disks in $\Sigma$ such that:
\begin{enumerate}

\item $\kappa \leq 0 \leq \lambda$, 
\item $\partial \delta_{i} = \alpha_{i}$, and

\item for $\kappa \leq i < \lambda$, the sphere $\delta_{i} \cup A_{i} \cup \delta_{i+1}$
bounds an embedded 3--ball, i.e., $\delta_{i}$ and $\delta_{i+1}$ are parallel rel $s$.

\end{enumerate}
See Figure~\ref{fig:c and delta} for an example illustrating this set-up and notation.

\begin{figure}[ht]
\begin{tikzpicture}
\node[inner sep=0pt] (a0) at (65:5 and 2) {};
\node[inner sep=0pt] (a1) at (-65:5 and 2) {};
\draw[thick,dashed,fill=red!30] (a0) to[out=-107,in=107] (a1) to[out=73,in=-73] (a0);
\draw[thick] (a0) to[out=-107,in=107] (a1);
\node[above] at (a0) {$\alpha_2$};
\node[inner sep=0pt] (b0) at (75:5 and 2) {};
\node[inner sep=0pt] (b1) at (-75:5 and 2) {};
\draw[thick,dashed,fill=red!30] (b0) to[out=-107,in=107] (b1) to[out=73,in=-73] (b0);
\draw[thick] (b0) to[out=-107,in=107] (b1);
\path (b0) -- (b1) node[pos=0.475] {$\delta_1$};
\node[above] at (b0) {$\alpha_1$};
\node[inner sep=0pt] (c0) at (85:5 and 2) {};
\node[inner sep=0pt] (c1) at (-85:5 and 2) {};
\draw[thick,dashed,fill=red!30] (c0) to[out=-107,in=107] (c1) to[out=73,in=-73] (c0);
\draw[thick] (c0) to[out=-107,in=107] (c1);
\node[above] at (c0) {$\alpha_0$};
\path (c0) -- (c1) node[pos=0.475] {$\delta_0$};
\node[inner sep=0pt] (d0) at (95:5 and 2) {};
\node[inner sep=0pt] (d1) at (-95:5 and 2) {};
\draw[thick,dashed,fill=red!30] (d0) to[out=-107,in=107] (d1) to[out=73,in=-73] (d0);
\draw[thick] (d0) to[out=-107,in=107] (d1);
\node[above] at (d0) {$\alpha_{-1}$};
\path (d0) -- (d1) node[pos=0.475] {$\delta_{-1}$};
\node[inner sep=0pt] (e0) at (105:5 and 2) {};
\node[inner sep=0pt] (e1) at (-105:5 and 2) {};
\draw[thick,dashed,fill=red!30] (e0) to[out=-107,in=107] (e1) to[out=73,in=-73] (e0);
\draw[thick] (e0) to[out=-107,in=107] (e1);
\node[above] at (e0) {$\alpha_{-2}$};
\path (e0) -- (e1) node[pos=0.475] {$\delta_{-2}$};
\node[inner sep=0pt] (f0) at (115:5 and 2) {};
\node[inner sep=0pt] (f1) at (-115:5 and 2) {};
\draw[thick,dashed,fill=red!30] (f0) to[out=-107,in=107] (f1) to[out=73,in=-73] (f0);
\draw[thick] (f0) to[out=-107,in=107] (f1);
\node[above] at (f0) {$\alpha_{-3}$};
\path (f0) -- (f1) node[pos=0.475] {$\delta_{-3}$};
\node[inner sep=0pt] (h0) at (125:5 and 2) {};
\node[inner sep=0pt] (h1) at (-125:5 and 2) {};
\draw[thick,dashed] (h0) to[out=-107,in=107] (h1) to[out=73,in=-73] (h0);
\draw[thick] (h0) to[out=-107,in=107] (h1);
\node[above] at (h0) {$\alpha_{-4}$};
\draw[thick] (4.1,0.4) ellipse (0.25 and 0.25);
\draw[thick] (4.1,0.4) ellipse (0.15 and 0.15);
\draw[thick] (4.1,-0.4) ellipse (0.25 and 0.25);
\draw[thick] (4.1,0) ellipse (0.45 and 0.75);
\draw[thick] (2.7,1.2) ellipse (0.2 and 0.2);
\draw[thick] (-3.4,1) ellipse (0.2 and 0.2);
\draw[thick] (-4.2,0.15) ellipse (0.2 and 0.2);
\draw[thick] (-4.2,0.15) ellipse (0.4 and 0.4);
\draw[thick] (-3.5,-1) ellipse (0.2 and 0.2);
\node[inner sep=0pt] (g0) at (50:5 and 2) {};
\node[inner sep=0pt] (g1) at (-50:5 and 2) {};
\draw[thick] (g0) to[out=-107,in=107] (g1);
\draw[thick,dashed] (g1) to[out=73,in=-73] (g0);
\filldraw[black] (1.55,1.2) -- (1.85,1.2) -- (1.85,1.5) -- (1.55,1.5) -- cycle;
\draw[|-|] (0.45,-2.2) -- (5,-2.2) node[pos=0.5,below] {$d_{+}$};
\draw[|-|] (0.35,-2.2) -- (-5,-2.2) node[pos=0.5,below] {$d_{-}$};
\draw[very thick,green] (0,0) ellipse (5 and 2);
\node at (5.3,0) {$s$};
\end{tikzpicture}
\caption{An example illustrating the curves $\{ \alpha_{i} \}$ and the disks $\{ \delta_{i} \}$.  The green sphere is $s$, its intersection with $\Sigma$ is in black and the red disks are the innermost disks in $\Sigma$.  The small black box represents an obstruction to isotoping the disk bounded by $\alpha_{2}$ to $\delta_{1}$ relative to $s$.  In this example $\kappa = -3$ and $\lambda = 1$.
}\label{fig:c and delta}
\end{figure}
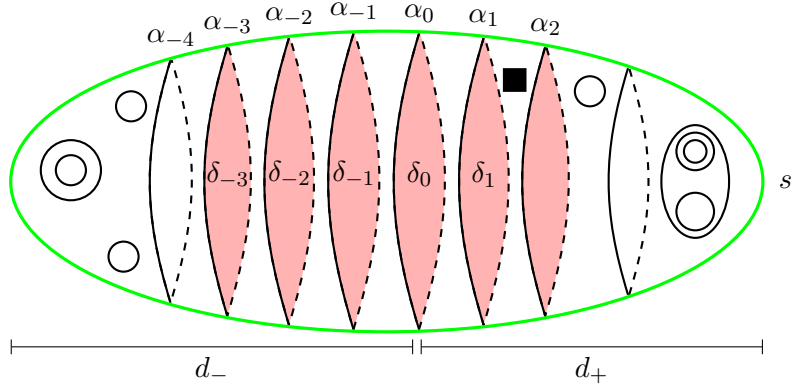

Using this set-up we can now describe a surgery of $S$.  Let $\delta_{-}$ be a parallel copy of $\delta_{\kappa}$ rel $s$ such that $\partial \delta_{-}$ and $\alpha_{\kappa}$ co-bound an annulus whose interior is disjoint from $\Sigma$ and $A_{\kappa}$.  Similarly let $\delta_{+}$ be a parallel copy of $\delta_{\lambda}$ rel $s$ such that $\partial \delta_{+}$ and $\alpha_{\lambda}$ co-bound an annulus whose interior is disjoint from $\Sigma$ and $A_{\lambda}$.  Set $\hat{d}_{-}$ to be the subdisk of $d_{-}$ with boundary $\partial \delta_{-}$ and set $\hat{d}_{+}$ to be the subdisk of $d_{+}$ with boundary $\partial \delta_{+}$  We get two new spheres $s_{-} = \hat{d}_{-} \cup \delta_{-}$ and $s_{+} = \hat{d}_{+} \cup \delta_{+}$.  We say that $\hat{S} = (S - \F_{n}\{ s\}) \cup \F_{n}\{ s_{+},s_{-} \}$ is obtained from $S$ by performing a \emph{surgery on $S$ with respect to $\Sigma$}.  

In what follows, it is important to record the history of the portions of the new spheres and so we introduce notion to this effect.  Suppose that $\hat{S}$ is the result of a surgery of $S$ with respect to $\Sigma$ and that $\hat{s} \in \hat{S}$ is (a translate of) one of the newly created spheres $s_{*} = \hat{d}_{*} \cup \delta_{*}$ for $* \in \{+,-\}$.  We call $d_{*}$ the \emph{portion of $\hat{s}$ from $S$}, denote it by $\hat{s}^{S}$.  Similarly, we call $\delta_{*}$ the \emph{portion of $\hat{s}$ from $\Sigma$}, denote it by $\hat{s}^{\Sigma}$.  Thus $\hat{s} = \hat{s}^{S} \cup \hat{s}^{\Sigma}$.  Notice that $\hat{s}^{S} \subseteq S$ and also that $\hat{s}^{\Sigma}$ is parallel rel $s$ to a disk in $\Sigma$.  For all other spheres $s \in \hat{S}$ we set $s^{S} = s$ and $s^{\Sigma} = \emptyset$. 

Our definition of surgery differs slightly from the standard in three ways: one, we do not remove parallel spheres in $\hat{S}$, two, we perform surgery along parallel innermost disks in a single step, and three, we do not isotope $S'$ to be in normal form with respect to $\Sigma$.  That we do not remove parallel spheres is in keeping with our definition of sphere systems from Section~\ref{sec:sphere-splittings}.  Justification of the latter two differences is the following lemma that shows that by performing surgery along the parallel innermost disks we can eliminate the need to perform a subsequent isotopy.  

\begin{lemma}\label{lem:normal after surgery}
Let $\hat{S}$ be the result of a surgery on $S$ with respect to $\Sigma$.  Then $\hat{S}$ and $\Sigma$ are in normal form.
\end{lemma}

\begin{proof}
Suppose otherwise.  As $S$ and $\Sigma$ are in normal form by assumption and normal form is a local condition, it must be that one of the newly created spheres is not in normal form with respect to $\Sigma$.  Denote this sphere by $\hat{s} = \hat{d} \cup \delta$ where $\hat{d}$ is a subdisk of the surgered sphere $s \in S$ and $\delta$ is a disk parallel rel $s$ to a disk in $\Sigma$.  Any intersection between $\hat{s}$ and some sphere  of $\Sigma$ must lie in $\hat{d} \subset s$ and hence $\hat{s}$ and a given sphere in $\Sigma$ intersect transversely in at most one circle as the same held for $s \in S$.

Therefore, if $\hat{s}$ is not in normal form with respect to $\Sigma$, then there is a sphere $\sigma \in \Sigma$ such that one of the disks in $\hat{s} - \sigma$, denote it $d$, is isotopic relative to its boundary to a disk in $\sigma$, denote it $\delta'$.  Without loss of generality, we can assume that this disk is innermost on $\hat{s}$, i.e., no subdisk of $d$ is isotopic relative to its boundary to a disk in some sphere of $\Sigma$.  The disk $d$ cannot lie entirely in $\hat{d}$ since $s$ and $\Sigma$ are in normal form by assumption.  Hence $d$ contains $\delta$.  Let $A$ be the annulus such that $d = A \cup \delta$.  Since $d \cup \delta'$ bounds a 3--ball, the assumptions that $S$ and $\Sigma$ are in normal form and that $d$ is innermost implies that $\Sigma$ is disjoint from the interior of $A$.  This contradicts the maximality assumption on the family of disks $\{\delta_{i}\}_{i = \kappa}^{\lambda}$.  Indeed, without loss of generality we can assume that $\delta = \delta_{+}$.  Then $A_{\lambda} \cup A$ is an annulus in $s$ whose interior is disjoint from $\Sigma$ and so $\partial \delta' = \alpha_{\lambda+1}$ and further $\delta_{\lambda} \cup (A_{\lambda} \cup A) \cup \delta'$ bounds an embedded 3--ball.    Hence $\hat{S}$ and $\Sigma$ are in normal form.
\end{proof}

\begin{definition}\label{def:surgery sequence}
A \emph{surgery sequence from $S$ to $\Sigma$} is a finite sequence of sphere systems: \[ S = S_{1}, \ldots, S_{m}  \] such that $S_{i+1}$ is the result of a surgery of $S_{i}$ with respect to $\Sigma$ and $d_{\calS}(\Sigma,S_m) \leq 1$.
\end{definition}

It is a standard fact that if $d_{\calS}(S,\Sigma) \geq 2$, then there is a surgery sequence from $S$, see for instance~\cite[Lemma~2.2]{un:HH}.  Further $d_{\calS}(S_{i},S_{i+1}) \leq 2$ as both $S_{i}$ and $S_{i+1}$ are dominated by $S_{i} \cup S_{i+1}$.

The discussion and notion regarding portions from $S$ and from $\Sigma$ make sense for surgery sequences as well by induction.  Indeed, suppose that $S_{i+1}$ is obtained from $S_{i}$ by a surgery with respect for $\Sigma$, specifically, assume that the (orbit of the) sphere $s \in S_{i}$ is split into (the orbit of) two spheres $s_{-} = \hat{d}_{-} \cup \delta_{-}$ and  $s_{+} = \hat{d}_{+} \cup \delta_{+}$ in $S_{i+1}$.  Then we have $s = \hat{d}_{-} \cup A \cup \hat{d}_{+}$ for some annulus $A$, the boundary curves of which are parallel to circles in $\Sigma$ rel $s$.  By choosing $A$ sufficiently narrow enough, we can assume that the annuli of $s$ witnessing the isotopy are contained in $s^{S}$.  We set $s_{*}^{S} = \hat{d}_{*} \cap s^{S}$ and $s_{*}^{\Sigma} = (\hat{d}_{*} \cap s^{\Sigma}) \cup \delta_{*}$ for $* \in \{ +,-\}$.  All other spheres in $S_{i+1}$ are also in $S_{i} - \{s\}$ and as such the portions from $S$ and $\Sigma$ remain unchanged.  See Figure~\ref{fig:portions}.  

\begin{figure}[ht]
\begin{tikzpicture}
\node[inner sep=0pt] (c0) at (80:5 and 2) {};
\node[inner sep=0pt] (c1) at (-80:5 and 2) {};
\draw[thick,dashed,fill=red!30] (c0) to[out=-107,in=107] (c1) to[out=73,in=-73] (c0);
\draw[thick] (c0) to[out=-107,in=107] (c1);
\node[above] at (c0) {$\alpha_+$};
\path (c0) -- (c1) node[pos=0.475] {$\delta_+$};
\node[inner sep=0pt] (d0) at (100:5 and 2) {};
\node[inner sep=0pt] (d1) at (-100:5 and 2) {};
\draw[thick,dashed,fill=red!30] (d0) to[out=-107,in=107] (d1) to[out=73,in=-73] (d0);
\draw[thick] (d0) to[out=-107,in=107] (d1);
\node[above] at (d0) {$\alpha_{-}$};
\path (d0) -- (d1) node[pos=0.475] {$\delta_{-}$};
\draw[thick,fill=red!30] (3.8,0) ellipse (0.45 and 0.75) node {$\delta_5$};
\draw[thick,fill=red!30] (2.3,0) ellipse (0.45 and 0.75) node {$\delta_4$};
\draw[thick,fill=red!30] (-4,0) ellipse (0.35 and 0.75) node {$\delta_1$};
\draw[thick,fill=red!30] (-3,0) ellipse (0.35 and 0.75) node {$\delta_2$};
\draw[thick,fill=red!30] (-2,0) ellipse (0.35 and 0.75) node {$\delta_3$};
\draw[|-|] (0.85,-2.2) -- (5,-2.2) node[pos=0.5,below] {$s_{+}$};
\draw[|-|] (-0.85,-2.2) -- (-5,-2.2) node[pos=0.5,below] {$s_{-}$};
\draw[very thick,green] (0,0) ellipse (5 and 2);
\node at (5.3,0) {$s$};
\end{tikzpicture}
\caption{An illustration showing a decomposition of $s \in S_{i}$ and the resulting spheres $s_{-}, s_{+} \in S_{i+1}$ into their portions from $S$ and $\Sigma$.  In this example $s^{\Sigma} = \{ \delta_{1},\delta_{2},\delta_{3},\delta_{4},\delta_{5} \}$, $s_{-}^{\Sigma} = \{\delta_{1},\delta_{2},\delta_{3},\delta_{-} \}$ and $s_{-}^{\Sigma} = \{\delta_{4},\delta_{5},\delta_{+} \}$.  The portions from $S$ are the complements in the respective spheres.}\label{fig:portions}
\end{figure}
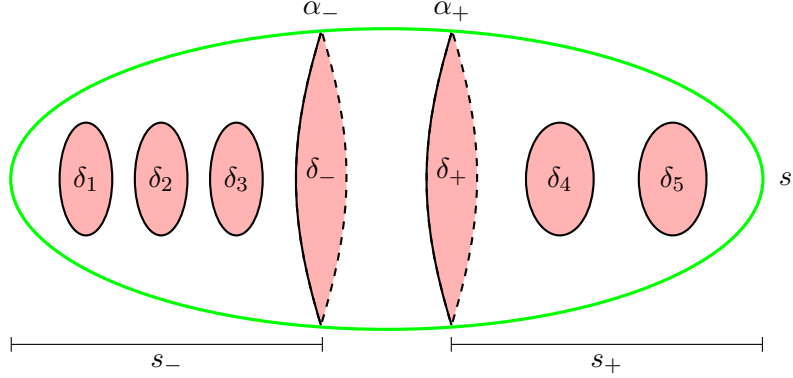

\begin{lemma}\label{Lem:Connected}
Suppose that $S = S_1, \ldots , S_m$ is a surgery sequence from $S$ to $\Sigma$.  Then for every $s \in S_i$, the subset $s^S$ is connected.
\end{lemma}

\begin{proof}
Using induction, we can conclude that the subset $s^{\Sigma}$ is a union of disks, each parallel rel $s$ to a disk in $\Sigma$.  Hence, $s^{S}$ is the complement of finitely many disks in $s$ and therefore connected.
\end{proof}

We remark that parallel spheres in $S_{i}$ may have different histories, that is, $s_{1},s_{2} \in S_{i}$ may be parallel even though $s_{1}^{S}$ and $s_{2}^{S}$ are not parallel.  For a surgery sequence $S = S_{1},\ldots,S_{m}$ from $S$ to $\Sigma$ we set
\begin{equation*}
S_{i}^{S} = \bigcup_{s \in S_{i}} s_{i}^{S} \mbox{ and } S_{i}^{\Sigma} = \bigcup_{s \in S_{i}} s_{i}^{\Sigma}.
\end{equation*}

\subsection{Rips moves and surgery steps}\label{subsec:rips}

Suppose that $S$ and $\Sigma$ are filling sphere systems and assume that they are in normal form.  Let $S = S_{1},S_{2},\ldots,S_{m}$ be a surgery sequence from $S$ to $\Sigma$ and let $\Sigma = \Sigma_{1},\Sigma_{2},\ldots,\Sigma_{\mu}$ be a surgery sequence from $\Sigma$ to $S$.  We describe $\Core(S_i, \Sigma_j)$ as (in some appropriate sense) an intersection of $\Core(S_i, \Sigma)$ and $\Core(S, \Sigma_j)$. We start by giving an embedding of $\Core(S_i, \Sigma)$ and $\Core(S, \Sigma_j)$ into $\Core(S, \Sigma)$.  This embedding is constructed inductively.  A single step in the construction is reminiscent to one of the elementary moves of the Rips machine~\cite{ar:BF95,ar:CH14}.  We start with a definition.

\begin{definition}\label{def:rips}
Suppose that $S$ and $\Sigma$ are filling sphere systems in $M$. We define $\partial_S$,
the $S$--boundary of $\Core(S,\Sigma)$, to be the subset of $\Core(S,\Sigma)$ 
consisting of the (open) edges $(P ,\sigma)$ that are the face of exactly one square 
or vertices $(P, \Pi)$ that are the vertex of exactly 3 edges of the form $(P ,\sigma)$, 
$(P ,\sigma')$ and $(s,\Pi)$. A connected component of $\partial_S$ is called 
an \emph{$S$--sides} of $\Core(S,\Sigma)$. Similarly, we define $\partial_\Sigma$,
the $\Sigma$--boundary of $\Core(S,\Sigma)$, to be the subset of $\Core(S,\Sigma)$ 
consisting of the (open) edges $(s ,\Pi)$ that are the face of exactly one square 
or vertices $(P, \Pi)$ that are the vertex of exactly 3 edges of the form $(s ,\Pi)$, 
$(s' ,\Pi)$ and $(P, \sigma)$. A connected component of $\partial_\Sigma$ is called 
an \emph{$\Sigma$--sides} of $\Core(S,\Sigma)$. 

The union of an $S$--side with the set of the (open) squares that have a face 
contained in that side is called a \emph{maximal $S$--boundary rectangle}. 
That is, in a $S$--maximal boundary rectangle, all of the squares are of the form 
$s_{0} \times \sigma$ for some fixed $s_{0} \in S$. A \emph{$\Sigma$--maximal boundary 
rectangle} is similarly defined from a connected component of the $\Sigma$--side.  
A \emph{Rips move} on $(S,\Sigma)$ is the removal of the $\F_{n}$--orbit of a 
($S$-- or $\Sigma$--)maximal boundary rectangle. 

If $R$ is a maximal boundary rectangle in $\Core(S,\Sigma)$, we let $\Core(S,\Sigma)_{R}$ denote the result of the associated Rips move.  We like to think of the removal of the maximal boundary rectangle as collapsing the rectangle by pushing across the adjacent squares.   
\end{definition}

We postpone presenting an example until after the following theorem. 

\begin{theorem}\label{thm:rips=surgery}
Suppose that $S$ and $\Sigma$ are filling sphere systems in $M$ and let $\hat{S}$ be the result of a surgery on $S$ with respect to $\Sigma$.  There is a $S$--maximal boundary rectangle $R \subseteq \Core(S,\Sigma)$ so that $\Core(S,\Sigma)_{R}$ is isomorphic to $\Core(\hat{S},\Sigma)$.  Moreover, for each $S$--maximal boundary rectangle, $R$, there is a sphere $\sigma \in \Sigma$ and innermost disk on $\sigma$ that defines a surgery $S \mapsto \hat{S}$ so that $\Core(\hat{S},\Sigma)$ is isomorphic to $\Core(S,\Sigma)_{R}$. 
%
\end{theorem}

\begin{proof}
Assume $\hat{S}$ is obtained from $S$ by a surgery on a sphere $s_{0} \in S$ and a disk $\delta$ that is part of the sphere system $\Sigma$, whose boundary 
$\alpha$ lies on $s$ and is otherwise disjoint from $S$.  By Lemma~\ref{lem:normal after surgery} $\hat{S}$ and $\Sigma$ are in normal form and so we can use the combinatorics of $\hat{S}$ and $\Sigma$ to build $\Core(\hat{S},\Sigma)$.  

We make use of the notation introduced in Section~\ref{subsection:surgery}.  Let $\{\delta_{i}\}_{i=\kappa}^{\lambda}$ be the maximal family of disks in $\Sigma$ parallel rel $s$ where $\delta_{0} = \delta$.  Let $A$ be the union of the annuli $A_{i} \subset s_{0}$ and $d_{+}$, $d_{-}$ the components of $s_{0} - A$.  Thus $\delta_{\kappa} \cup A \cup \delta_{\lambda}$ bounds a 3--ball $B$.  The two spheres obtained by surgery of $s$ using this family, $s_{+}$ and $s_{-}$, are parallel to $d_{+} \cup \delta_{\lambda}$ and $d_{-} \cup \delta_{\kappa}$ respectively.

Let $P_{+} \in M - S$ be the component that contains the interior of $B$ and let $P_{-}$ be the other component with $s$ as a boundary.  Each disk $\delta_{i}$ is contained in some sphere $\sigma_{i} \in \Sigma$.  For each $\kappa \leq i < \lambda$, there are components $\Pi_{i} \subset M - \Sigma$ such that both $\sigma_{i}, \sigma_{i+1} \subset \partial \Pi_{i}$.  We claim that the collections of edges and vertices:  
\[ (P_{+},\sigma_{\kappa}), (P_{+},\Pi_{\kappa}), (P_{+},\sigma_{\kappa+1}), \ldots, (P_{+},\sigma_{\lambda})\]
is a side.  Indeed, each edge $(P_{+},\sigma_{i})$ is the face of only $s \times \sigma_{i}$ and each vertex $(P_{+},\Pi_{i})$ is only also adjacent to $(s_{0},\Pi_{i})$.  The first of these observations is due to the fact that $P_{+} \cap \sigma_{i} = \delta_{i}$ is a disk; the second observation due to the fact that $P_{+} \cap \Pi_{i}$ is bounded by $\delta_{i} \cup A_{i} \cup \delta_{i+1}$.  Maximality of this collection follows from maximality of the collection $\{\delta_{i}\}_{i=\kappa}^{\lambda}$.  

Let $R$ be the corresponding maximal boundary rectangle of $\Core(S,\Sigma)$.  We will show that $\Core(\hat{S},\Sigma)$ is isomorphic to $\Core(S,\Sigma)_{R}$.  To do so, we will construct an injection of square complexes $\Core(\hat{S},\Sigma) \hookrightarrow \Core(S,\Sigma)$ whose image is $\Core(S,\Sigma)_{R}$.
 
Components in $M - S$ that are not in the orbit of $P_{+}$ and $P_{-}$ are also components of $M - \hat{S}$.  But $M - \hat{S}$ has 3 other components; $\hat{P}_{-}$ which is obtained from $P_-$ by adding a neighborhood of $s$ and a neighborhood of the 3--ball $B$ bounded by $\delta_{\kappa} \cup A \cup \delta_{\lambda}$, $P_+^+$ and $P_+^-$ which are contained in the two components of $P_{+} - B$.  In other words, we have: 
\[
M - \hat{S} = \bigl((M - S) - \F_{n}\{P_{+},P_{-}\}\bigr) \cup  \F_{n}\{ \hat{P}_-, P_+^+, P_+^- \big\}. 
\]
There is an $\F_{n}$--equivariant map $\iota \from M - \hat{S} \to M - S$ defined by $P_{+}^{+},P_{+}^{-} \mapsto P_{+}$, $\hat{P}_{-} \mapsto P_{-}$ and the identity on the other orbits.  Also, there is a $\F_{n}$--equivariant map $\epsilon \from \hat{S} \to S$ defined by $s_{+},s_{-} \mapsto s_{0}$ and the identity on the other orbits.      

Using $\iota$, we get a map on the 0--skeleton of $\Core(\hat{S},\Sigma)$ defined by $(P,\Pi) \mapsto (\iota(P),\Pi)$.  In order for this to be well-defined, we need to know that if $P \cap \Pi \neq \emptyset$, then $\iota(P) \cap \Pi \neq \emptyset$ also.  If $P$ is not in the orbit of $\hat{P}_{-}$, then this follows as $P \cap \Pi \subseteq \iota(P) \cap \Pi$.  Finally, since $\hat{P}_{-} = P_{-} \cup B$ and no component of $M - \Sigma$ is contained in $B$, any component of $M - \Sigma$ that intersects $\hat{P}_{-}$ necessarily intersects $P_{-}$ as well.

We extend over the $1$--skeleton using $\epsilon$: $(s,\Pi) \mapsto (\epsilon(s),\Pi)$.  This map is well-defined since any intersection between $s_{+}$, or $s_{-}$, with a component of $M - \Sigma$ is contained in the portion of $s_{+}$, or $s_{-}$ respectively, from $S$, i.e., $d_{+}$, or $d_{-}$ respectively.  Notice that this is consistent with the mapping on the $0$--skeleton.  The edge $(s_{+},\Pi)$ in $\Core(\hat{S},\Sigma)$ is sent to $(s_{0},\Pi)$.  The vertices of $(s_{0},\Pi)$ are $(P_{+}^{+},\Pi)$ and $(\hat{P}_{-},\Pi)$, which are the images of $(P_{+}^{+},\Pi)$ and $(\hat{P}_{-},\Pi)$.  Other verifications are similar.

Finally, we extend over the $2$--skeleton: $s \times \sigma \mapsto \epsilon(s) \times \sigma$.  Since any intersection of $s_{+}$ or $s_{-}$ with $\Sigma$ is contained in the portion from $S$, this map is well-defined.  Again, the map on the $2$--skeleton is consistent with the maps on  $1$--skeleton and $0$--skeleton by construction.  
  
The map $\Core(\hat{S},\Sigma) \to \Core(S,\Sigma)$ is $\F_{n}$--equivariant and preserves the square structure.   The map is not surjective as no $2$--cell is mapped to the squares associated with $s \times \sigma_{i}$, i.e., the image of the map is exactly $\Core(S,\Sigma)_{R}$.   
 
The converse is similar, if $s_{0} \times \sigma_{1},\ldots, s_{0} \times \sigma_{\lambda}$ forms an $S$--maximal boundary rectangle $R$, then one shows that there are disks $\delta_{i} \subset \sigma_{i}$ that are parallel rel $s_{0}$ and that surgery using the family $\{\delta_{i}\}_{i=1}^{\lambda}$ results in the sphere system $\hat{S}$ where $\Core(\hat{S},\Sigma)$ is isomorphic to $\Core(S,\Sigma)_{R}$.
\end{proof}

\begin{example}\label{ex:rips}
Here we describe a Rips move and the corresponding surgery explicitly in an example.  
Consider a sphere $s \in S$ associated to the edge $b$ in the dual graph $G$. In
the example depicted in Figure~\ref{fig:rips}, $s$ intersects $7$ spheres
in $\Sigma$; spheres $\sigma_0, \ldots, \sigma_6$ associated to edges
$\eta_0, \ldots, \eta_6$ in $\Gamma$. We denote the intersection circle between 
$s$ and $\sigma_i$ by $\alpha_i$. The slice over $s$ in $\Core(S, \Sigma)$ 
consists of squares associated to intersection circles between $s$ and $\Sigma$
that is, 
\[
C_{s} = \{ s \times \sigma_i \mid i = 0, \ldots, 6 \}.
\]  
In the language of trees, $C_s$ is associated to the slice over
$b$, which is 
\[
b \times \Shadow(s) = \{ b \times \eta_i \mid i = 0, \ldots, 6 \}
  \subset G \times \Gamma.
\]

There are two components of $M-S$ that have $s$ as their boundary sphere.
In this example, the component $P_l$, which we call left, has 3 other boundary 
spheres (associated to edges $a_1$, $a_2$ and $a_3$) and the component $P_r$ 
on the right has two other boundary spheres (associated to edges $c_1$ and $c_2$). 
 
\begin{figure} [ht]
\includegraphics{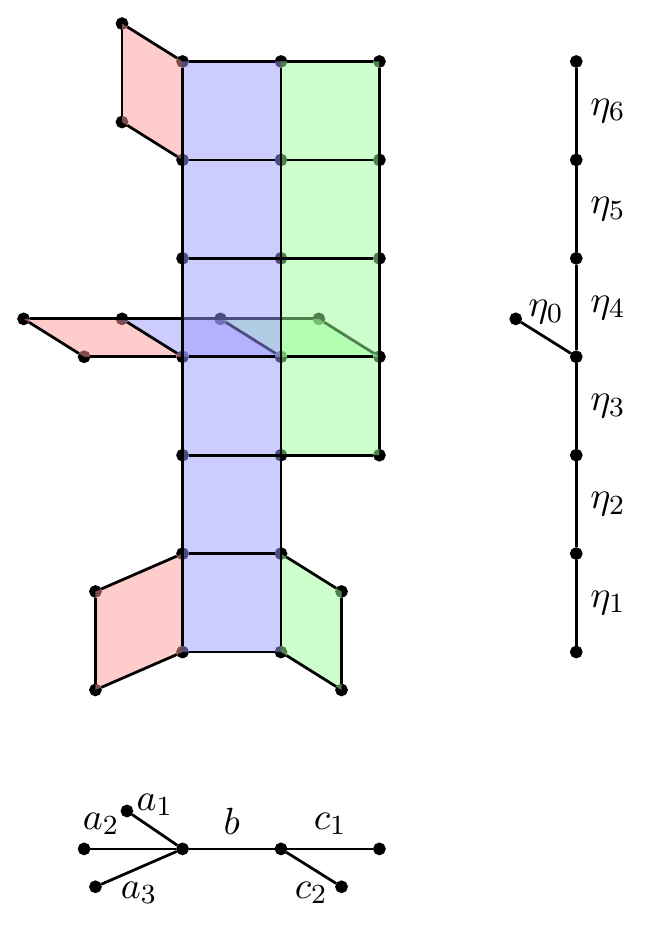}
\includegraphics{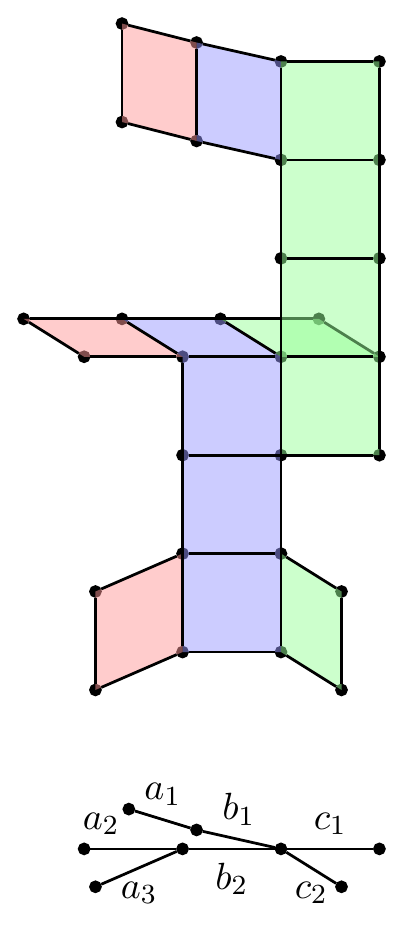}
\caption{The left hand side depicts the slice $C_s$ and squares attached to it 
in $\Core(S, \Sigma)$. Consider the maximal $S$--boundary rectangle 
$R= b \times (\eta_4 \cup \eta_5)$. 
A Rips move along $R$ is associated to a surgery on the sphere $s$
or a splitting of the edge $b$ in the graph $G$. The right side depicts the 
associated portion of $\Core(S, \Sigma)_R=\Core(S',R)$. }
\label{fig:rips}
\end{figure}

Note that Figure~\ref{fig:rips} indicates that the sphere $\sigma_1$ intersects spheres 
in $S$ associated to edges $a_1$, $b$ and $c_2$ since the core contains squares
$a_1 \times \eta_1$, $b \times \eta_1$ and $c_2 \times \eta_1$. However, the
sphere $\sigma_2$ does not intersect spheres associated to edges $a_1$, $a_2$ and
$a_3$. But, $\sigma_2$ intersects $s$, hence, the circle $\alpha_2$ must bound a disk 
$\delta_2$ that is the intersection of $\sigma_2$ with $P_l$. 
Similarly, circles $\alpha_3$, $\alpha_4$ and $\alpha_5$ bound disks
$\delta_3$, $\delta_4$ and $\delta_5$ that are, respectively, intersections of 
spheres $\sigma_3$, $\sigma_4$ and $\sigma_5$ with $P_l$,  
(thus the squares $b \times \eta_i$, $i=2, \ldots, 5$, have boundary edges on their 
left side).  The circle $\alpha_2$ also bounds a disk in $\sigma_2$ in $P_r$ 
(thus the square $b \times \eta_2$ has a boundary edge on its right side). 

The disks $\delta_2$ and $\delta_3$ are parallel and the disks $\delta_4$ and $\delta_5$ 
are also parallel, however, the two sets of disks are not parallel to each other
(see Figure~\ref{Fig:circles}). Thus, there are two maximal boundary rectangles 
from the left; $R= b \times(\eta_2 \cup \eta_3)$ and $R'=b \times (\eta_4 \cup \eta_5)$. 
More precisely, let $\Pi$ be the component of $M -\Sigma$
with $\sigma_0$, $\sigma_3$ and $\sigma_4$ as its boundary spheres. 
Then, referring to Definition~\ref{def:rips}, we see that the vertex
$(P_l,\Pi)$ is not in the $S$--boundary of $\Core(S, \Sigma)$ because it is the vertex of 
$5$ different edges. Hence, the union of $R$ and
$R'$ is not a boundary rectangle.

\begin{figure} [ht]
\includegraphics{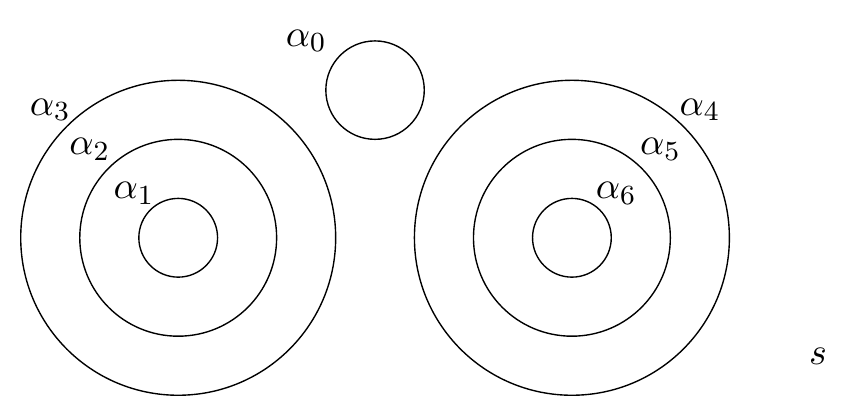}
\caption{The curves $\alpha_0, \ldots, \alpha_6$ are intersection circles between 
the sphere $s$ and the spheres $\sigma_0, \ldots, \sigma_6$ respectively. 
The circles $\alpha_2$, $\alpha_3$, $\alpha_4$ and $\alpha_5$
bound a disk in $P_l$. However, the circle $\alpha_3$ and $\alpha_4$ are not parallel.}
\label{Fig:circles}
\end{figure}

Define $S'$ to be the sphere system obtained from $S$ by applying the surgery on the set 
of parallel disks $\{ \delta_4, \delta_5 \}$ (and their $\F_n$--orbits). The surgery results in 
two spheres $s_1$ and $s_2$ associated 
to edges  $b_1$ and $b_2$ and the removal of the maximal boundary rectangle $R$.  
It appears that removal of this rectangle makes the slice over $b$ disconnected. 
However, the two components are slices over the edges $b_1$ and $b_2$. 

To summarize, the surgery along the disks $\{ \delta_4, \delta_5\}$, changes $G$ by splitting the edge $b$ and changes $\Core(S, \Sigma)$ by removing 
the maximal boundary rectangle $R = b \times (\eta_4 \cup \eta_5)$, resulting in $\Core(S,\Sigma)_{R} \cong \Core(S',\Sigma)$. 

A different splitting of $b$ into $b_{1}$ and $b_{2}$ partitioning the $a$ edges into $\{a_{1},a_{3}\}$ and $\{a_{2}\}$ does not arise as a surgery and could potentially increase the volume of the core.
\end{example}

\subsection{The intersection of cores}\label{subsec:intersection}

Applying Theorem~\ref{thm:rips=surgery} to the surgery sequence $S = S_{1},S_{2},\ldots,S_{m}$ we obtain maps for $i = 1,\ldots,m-1$:
\[
k_{i, i+1}: \Core(S_{i+1}, \Sigma) \to \Core(S_i, \Sigma).
\] that are the composition of the isomorphism $\Core(S_{i+1},\Sigma) \cong \Core(S_{i},\Sigma)_{R}$ for the corresponding maximal boundary rectangle and the natural inclusion $\Core(S_{i},\Sigma)_{R} \hookrightarrow \Core(S_{i},\Sigma)$.  By symmetry there are also maps for $j = 1,\mu-1$:
\[\kappa_{j,j+1} \from  \Core(S, \Sigma_{j+1}) \to \Core(S, \Sigma_j).\]  
Since these map exist for all $1 \leq i \leq m-1$, $ 1 \leq j \leq \mu-1$, we can define ``inclusions'' alluded to at the beginning of this section:
\begin{align}
k_i &= k_{1,2}k_{2,3}...k_{i-2, i-1}k_{i-1, i}\from \Core(S_i,\Sigma) \to \Core(S, \Sigma)\label{eq:ki} \\
\kappa_j &= \kappa_{1,2}\kappa_{2,3}...\kappa_{j-2, j-1}\kappa_{j-1,j}\from \Core(S,\Sigma_j) \to \Core(S, \Sigma)\label{eq:kappaj}
\end{align}

\begin{remark}\label{rem:rips}
On the level of squares, the map $k_{i}\from \Core(S_{i},\Sigma) \to \Core(S,\Sigma)$ is easy to describe.  For each $\hat{s} \in S_{i}$, we have that $\hat{s}^{S} \subseteq s$ for a unique $s \in S$.  The map is defined by $\hat{s} \times \sigma \to s \times \sigma$.
\end{remark}

The following is the fundamental concept essential to the proof of the main theorem.


\begin{proposition}\label{Prop:Isomorphism} 
With the above set-up, assume 
\begin{equation} \label{eq:Union} 
k_i\bigl(\Core(S_i, \Sigma)\bigr) \cup \kappa_j\bigl(\Core(S, \Sigma_j)\bigr) 
= \Core(S, \Sigma),
\end{equation}
then $S_i$ and $\Sigma_j$ are in normal form. Furthermore, there exists an  
isomorphism 
\[ 
\Phi \from \Core(S_i, \Sigma_j) \to 
   k_i\bigl(\Core(S_i, \Sigma)\bigr) \cap \kappa_j\bigl(\Core(S, \Sigma_j)\bigr). 
\] 
\end{proposition} 
 
\begin{proof}
First, we show that every intersection circle between $S_i$ and $\Sigma_j$
is in fact in $S_i^S \cap \Sigma_j^\Sigma$. This is because a square in 
$\Core(S,\Sigma)$ associated to an intersection circle in $S_i^\Sigma \cap \Sigma_j^S$ is neither in $k_i\bigl(\Core(S_i, \Sigma)\bigr)$ ($S_i^\Sigma$ does not intersect $\Sigma$) nor in $\kappa_j\bigl(\Core(S, \Sigma_j)\bigr)$ ($\Sigma_j^S$ does not intersect $S$) and by the assumption \eqref{eq:Union} every square in $\Core(S,\Sigma)$ is in the image of one of these two maps.

This observation implies that $S_i$ and $\Sigma_j$ are in fact in normal form. 
In fact, pick spheres $s_i \in S_i$ and $\sigma_j \in \Sigma_j$. We will show
that $s_i$ and $\sigma_j$ intersect at most once. Otherwise, 
$s_i^S$ and $\sigma_j^\Sigma$ intersect more than once. But, by 
\lemref{Lem:Connected}, $s_i^S$ and $\sigma_j^\Sigma$ are connected, which means 
there is a sphere $s \in S$ that contains $s_i^S$ and a sphere $\sigma \in \Sigma$ that contains $\sigma_j^\Sigma$. Hence, $s$ and $\sigma$ intersect more than once. This contradicts the fact that $S$ and $\Sigma$ are in normal form. 

Now consider a square $s_{i} \times \sigma_{j}$ in $\Core(S_i, \Sigma_j)$ associated to an intersection circle $\alpha$.  Then $\alpha$ is an intersection circle in $S_i^S \cap \Sigma_j^\Sigma$.  Which mean it is an intersection circle in both $S \cap \Sigma_j$ and $S_i \cap \Sigma$ and thus there are spheres $s \in S$, $\sigma \in \Sigma$ for which $s \cap \sigma = \alpha$ and $s_{i}^{S} \subseteq s$, $\sigma_{j}^{\Sigma} \subseteq \sigma$.  Hence, $s \times \sigma$ is 
contained in both $k_i\bigl(\Core(S_i, \Sigma)\bigr)$ and  
$\kappa_j\bigl(\Core(S, \Sigma_j)\bigr)$ and so we define $\Phi(s_{i} \times \sigma_{j}) = s \times \sigma$.  Normal form implies that the map is injective. 

To prove that $\Phi$ is surjective, suppose $s \times \sigma$ is in 
$k_i\bigl(\Core(S_i,\Sigma)\bigr)$.  Then the associated intersection circle in
$S_i^S$.  Similarly, assuming $s \times \sigma$ is in $\kappa_j\bigl(\Core(S, \Sigma_j)\bigr)$ implies that the associated intersection circle in $\Sigma_j^\Sigma$.  Therefore, it also lies in $S_i \cap \Sigma_j$.  Hence there are spheres $s_{i} \in S_{i}$ and $\sigma_{j} \in \Sigma_{j}$ such that $\Phi(s_{i} \times \sigma_{j}) = s \times \sigma$. 
\end{proof}

For future reference, we record the following corollary:

\begin{corollary} \label{Cor:Surgery}
As long as \eqref{eq:Union} is satisfied, $\Core(S_i, \Sigma_{j+1})$ can obtained from $\Core(S_i, \Sigma_j)$ by a Rips move.
\end{corollary}

\begin{proof}
Let $R$ be a maximal $\Sigma_{j}$--boundary rectangle in $\Core(S,\Sigma_{j})$ such that $\Core(S,\Sigma_{j})_{R} \cong \Core(S,\Sigma_{j+1})$.  Thus $R$ consists of squares $s_{1} \times \hat{\sigma}, \ldots, s_{\ell} \times \hat{\sigma}$ for some $\hat{\sigma} \in \Sigma_{j}$ and $s_{1},\ldots,s_{\ell} \in S$.  Let $\sigma \in \Sigma$ be such that $\hat{\sigma}^{\Sigma} \subseteq \sigma$ and consider the set $C_{\sigma,i}$ of squares of the form $\hat{s} \times \sigma$ in $\Core(S_{i},\Sigma)$.  Then $k_{i}(C_{\sigma,i}) \cap \kappa_{j}(R)$ corresponds via the isomorphism in Proposition~\ref{Prop:Isomorphism} to a maximal $\Sigma_{j}$--boundary rectangle in $\Core(S_{i},\Sigma_{j})$ whose collapse results in $\Core(S_{i},\Sigma_{j+1})$.
\end{proof}

%% file: proof.tex
\section{Proof of Theorem~\ref{Thm:Main}}\label{Sec:Proof} 

To finish the proof of the main theorem, we proceed as follows using the set-up from the previous section.  We start with a lemma giving a necessary condition for two sphere systems to be at a bounded distance.  A \emph{free edge} is an edge that does not bound any squares.

\begin{lemma}\label{Lem:Isolated}
If $\Core(S_i, \Sigma_j)$ contains a free edge, then the two sphere systems are of distance at most 2 in the sphere graph.
\end{lemma}

\begin{proof}
Edges in the core are associated to spheres in either sphere system $S_i$ or $\Sigma_j$ and squares are associated to intersection circles between sphere systems. Hence, a free edge in the core is associated to a sphere in either $S_i$ or $\Sigma_j$ that does not intersect any other spheres from the other system. Thus this sphere can be added to both sphere systems.  That is, $\bS_i$ and $\bSigma_j$ have distance 2 in the sphere graph.
\end{proof}

We now prove Theorem~\ref{Thm:Main}.  We restate it for convenience.

\medskip \noindent {\bf Theorem 1.1.~~}{\it Let $\bS$ and $\bSigma$ be two filling sphere systems and let 
\begin{equation*}
\bS= \bS_1, \bS_2, \ldots, \bS_m,  
\qquad  d_{\calS}(\bS_{m},\bSigma) \leq 1
\end{equation*}
be a surgery sequence starting from $\bS$ towards $\bSigma$ and 
\begin{equation*}
\bSigma= \bSigma_1, \bSigma_2, \ldots, \bSigma_{\mu},
\qquad d_{\calS}(\bSigma_{\mu},\bS) \leq 1
\end{equation*}
be a surgery sequence in the opposite direction. Then, for every $\bS_i$ there is a 
$\bSigma_j$ so that $d_\calS(\bS_i, \bSigma_j) \leq 2$.  
}

\begin{proof}
Fix two filling sphere systems $\bS$ and $\bSigma$ and surgery paths as in the statement of the theorem.  For every $\bS_i$ we need to find $\bSigma_j$ with $d_\calS(\bS_i, \bSigma_j) \leq 2$.  Fix an $i = 1, \ldots ,m$ and let $j$ be the largest index where the equality 
\begin{equation}\label{eq:Equal} 
k_i\bigl(\Core(S_i, \Sigma)\bigr) \cup \kappa_j\bigl(\Core(S, \Sigma_j)\bigr) = \Core(S, \Sigma)
\end{equation}
still holds.  Note that the equation holds when $j =1$. But, since $\kappa_j\bigl(\Core(S, \Sigma_j)\bigr)$ eventually contains no squares (for instance, when $j = \mu$) and $k_i\bigl(\Core(S_i, \Sigma)\bigr)$ is a proper subset of $\Core(S, \Sigma)$ for each $i > 1$, there exists an index $j+1$ for which \eqref{eq:Equal} does not hold. 

We will show that $\Core(S_i, \Sigma_{j})$ contains a free edge.  By Lemma~\ref{Lem:Isolated}, this will complete the proof.  Let $s \times \sigma$ be a square in $\Core(S, \Sigma)$ that is not contained in $k_i\bigl(\Core(S_i, \Sigma)\bigr) \cup \kappa_{j+1}\bigl(\Core(S, \Sigma_{j+1})\bigr)$.  By \eqref{eq:Equal}, $s \times \sigma$ is contained in $\kappa_j\bigl(\Core(S, \Sigma_j)\bigr)$.   Thus a surgery on $\Sigma_j$ has deleted the intersection circle associated to this square. By \corref{Cor:Surgery}, $s \times \sigma$ is part of a maximal $\Sigma_{j}$--boundary rectangle.  That is, there is a component $\Pi \subseteq M - \Sigma$ for which $\sigma \in \partial \Pi$ so that the edge $(s,\Pi)$ is a boundary edge of $s \times \sigma$ but not the boundary edge of any other square in $\kappa_j\bigl(\Core(S, \Sigma_j)\bigr)$.  

We also know that $s \times \sigma$ is not contained in $k_i\bigl(\Core(S_i, \Sigma)\bigr)$.  Thus, if $(s,\Pi)$ is an edge in $k_{i}\bigl(\Core(S_{i},\Sigma)\bigr)$ then we have that $(s,\Pi)$ is a free edge in $k_{i}\bigl(\Core(S_{i},\Sigma)\bigr) \cap \kappa_{j}\bigl(\Core(S,\Sigma_{j})\bigr) \cong \Core(S_{i},\Sigma_{j})$ (Proposition~\ref{Prop:Isomorphism}).  If this is not the case, then there is some $i_{0} < i$ such that $(s,\Pi)$ lies between two squares $s \times \sigma'$ and $s \times \sigma''$ that are part of a maximal $S_{i_{0}}$--boundary rectangle in $\Core(S_{i_{0}},\Sigma)$ that is collapsed in the formation of $\Core(S_{i_{0}+1},\Sigma)$.  Then neither of these squares are in $k_{i}\bigl(\Core(S_{i},\Sigma)\bigr)$
at least one of these squares is not in $\kappa_{j}\bigl(\Core(S,\Sigma_{j})\bigr)$.  However, this would contradict \eqref{eq:Equal}.  Therefore, $(s,\Pi)$ is a free edge in $\Core(S_i, \Sigma_{j})$.
\end{proof}



%